\documentclass[a4paper,11pt,reqno,smallextended]{amsart} 

\usepackage{latexsym}
\usepackage{mathrsfs}
\usepackage{stmaryrd}
\usepackage[vcentermath]{youngtab}
\newcommand{\yngt}{\youngtabloid}
\newcommand{\tyoung}[1]{\scalebox{0.8}{$\displaystyle\young#1$}}
\newcommand{\syoung}[1]{\scalebox{0.75}{$\displaystyle\young#1$}}
\newcommand{\ssyoung}[1]{\scalebox{0.9}{$\displaystyle\young#1$}}
\newcommand{\sssyoung}[1]{\scalebox{0.95}{$\displaystyle\young#1$}}
\usepackage{amssymb,amsmath,amsopn}
\usepackage{graphics}
\usepackage{relsize}
\usepackage{mathabx}

\usepackage{makecell}
\makeatletter
\newcommand{\leqnomode}{\tagsleft@true}
\newcommand{\reqnomode}{\tagsleft@false}
\makeatother

\usepackage{url}

\usepackage{tikz}

\newtheorem{theorem}{Theorem}[section]
\newtheorem{lemma}[theorem]{Lemma}

\newtheorem{corollary}[theorem]{Corollary}

\newtheorem{proposition}[theorem]{Proposition}

\theoremstyle{definition}
\newtheorem{definition}[theorem]{Definition}
\newtheorem{example}[theorem]{Example}
\newtheorem{algorithm}[theorem]{Algorithm}

\DeclareMathOperator{\GL}{GL}
\DeclareMathOperator{\SL}{SL}
\DeclareMathOperator{\sgn}{sgn}
\DeclareMathOperator{\exw}{Ext}

\DeclareMathOperator{\Hom}{Hom}

\DeclareMathOperator{\Stab}{Stab}

\DeclareMathOperator{\Sym}{Sym}
\DeclareMathOperator{\Inf}{Inf}

\newcommand{\CS}{\mathrm{CS}}

\newcommand{\e}{\mathrm{e}}

\renewcommand{\theta}{\vartheta}

\newcommand{\N}{\mathbb{N}}

\newcommand{\Q}{\mathbb{Q}}

\newcommand{\Ind}{\big\uparrow}

\newcommand{\ind}{\!\!\uparrow}
\newcommand{\res}{\!\!\downarrow}

\newcounter{thmlistcnt}
\newenvironment{thmlist}%
	{\setcounter{thmlistcnt}{0}%
	\begin{list}{\emph{(\roman{thmlistcnt})}}{%
		\usecounter{thmlistcnt}%
		\setlength{\topsep}{0pt}%
		\setlength{\leftmargin}{24pt}%
		\setlength{\itemsep}{0pt}%
		\setlength{\labelwidth}{24pt}
		\setlength{\itemindent}{0pt}}%
	}%
	{\end{list}}%
	
\newcounter{defnlistcnt}
\newenvironment{defnlist}%
	{\setcounter{defnlistcnt}{0}%
	\begin{list}{(\roman{defnlistcnt})}{%
		\usecounter{defnlistcnt}%
		\setlength{\topsep}{0pt}%
		\setlength{\leftmargin}{32pt}%
		\setlength{\itemsep}{0pt}%
		\setlength{\labelwidth}{32pt}
		\setlength{\itemindent}{0pt}}%
	}%
	{\end{list}}%

\newcommand{\join}{\sqcup}

\newcommand{\pp}{\cdot}

\newcommand{\indmne}{\,\ind_{S_3 \wr S_8}^{S_{24}}}

\newcommand{\indmn}{\,\ind_{S_m \wr S_n}^{S_{mn}}}
\newcommand{\Indmn}{\Ind_{S_m \wr S_n}^{S_{mn}}}

\newcommand{\RPP}{\mathrm{RPP}}
\newcommand{\CPP}{\mathrm{CPP}}

\renewcommand{\sp}{\,,\ \rule[-12pt]{0pt}{30pt}}
\newcommand{\hs}{\hspace*{0.1in}}
\newcommand{\hsss}{\hspace*{0.02in}}

\newcommand{\id}{\mathrm{id}}
\newcommand{\bT}{\mathbf{T}}
\newcommand{\bS}{\mathbf{S}}
\newcommand{\bU}{\mathbf{U}}
\newcommand{\bV}{\mathbf{V}}

\newcommand{\smallCalT}{\scalebox{0.8}{$\scriptstyle\mathcal{T}$}}
\newcommand{\bTT}{\bT_{\smallCalT}}
\newcommand{\vsmallCalT}{\scalebox{0.5}{$\scriptstyle\mathcal{T}$}}

\newcommand{\bTTs}{\bT_{\vsmallCalT}}

\newcommand{\D}{D} 
\newcommand{\Sind}{\mathbf{S}}
\newcommand{\timesn}{n}
\newcommand{\f}{f_\mathcal{T}}
\newcommand{\g}{g_\mathcal{T}}
\newcommand{\overf}{\overline{f}_\mathcal{T}}

\newcommand{\overg}{\overline{g}_\mathcal{T}}

\newcommand{\NS}{\mathcal{N}\hskip-0.5pt\mathcal{S}}
\newcommand{\NR}{\mathcal{N}\hskip-0.5pt\mathcal{R}}

\newcommand{\oa}{1_1}
\newcommand{\ob}{1_2}
\newcommand{\oc}{1_3}
\newcommand{\od}{1_4}
\renewcommand{\oe}{1_5}
\newcommand{\of}{1_6}
\newcommand{\og}{1_7}
\newcommand{\oh}{1_8}
\newcommand{\ta}{2_1}
\newcommand{\tb}{2_2}
\newcommand{\tc}{2_3}
\newcommand{\td}{2_4}
\newcommand{\te}{2_5}
\newcommand{\da}{3_1}
\newcommand{\db}{3_2}

\newcommand{\onep}{1^\star}
\newcommand{\twop}{2^\star}
\newcommand{\threep}{3^\star}
\newcommand{\fourp}{4^\star}
\newcommand{\fivep}{5^\star}
\newcommand{\sixp}{6^\star}

\newcommand{\row}{\mathrm{row}}
\newcommand{\col}{\mathrm{col}}
\renewcommand{\i}{\mathrm{disj}}
\newcommand{\entry}{\mathrm{entry}}

\newcommand{\maj}{\mathrm{maj}}

\newcommand{\tto}{t_1}
\newcommand{\ttt}{t_2}
\newcommand{\ttd}{t_3}
\newcommand{\ttf}{t_5}
\newcommand{\tts}{t_6}
\newcommand{\ttn}{t_9}
\newcommand{\exu}{u}
\renewcommand{\exw}{w}
\newcommand{\exv}{v}
\newcommand{\exx}{x}

\newcommand{\erow}{e_\row}
\newcommand{\ecol}{e_\col}

\newcommand{\R}{\mathcal{P}}

\newcommand{\nuOmega}{{\Omega_\nu}}
\newcommand{\Gnu}{{G_\nu}}

\newcommand{\bUT}{\bT}
\newcommand{\bVV}{\bV}


\author{Rowena Paget}
\address[R. Paget]{School of Mathematics, Statistics and Actuarial Science, University of Kent,
Canterbury, CT2 7NF, U.K.}
\email{R.E.Paget@kent.ac.uk}

\author{Mark Wildon}
\address[M. Wildon]{Department of Mathematics, Royal Holloway, University of London, Egham, TW20 0EX,
U.K.}
\email{mark.wildon@rhul.ac.uk}

\begin{document}

\thispagestyle{empty}

\title[Generalized Foulkes modules]{Generalized Foulkes modules and
maximal and minimal constituents of plethysms of Schur functions}


\date{September 2018}
\maketitle
\thispagestyle{empty}

\begin{abstract}
This paper proves a combinatorial rule giving all maximal and minimal partitions $\lambda$ such
that the Schur function $s_\lambda$ appears in a plethysm of two arbitrary Schur functions.
Determining the decomposition of these plethysms
 has been identified by Stanley as a key open problem in algebraic combinatorics.
As corollaries we prove three conjectures of Agaoka on the partitions labelling
the lexicographically
greatest and least Schur functions appearing in an arbitrary plethysm. We also show that the multiplicity
of the Schur function labelled by the lexicographically least  constituent may
be arbitrarily large.
The proof is carried out in the symmetric group and gives an explicit non-zero homomorphism
corresponding to each maximal or minimal partition.
\end{abstract}

\section{Introduction}

Fix $m$, $n \in \N$ and let $S_m \wr S_n$ be the wreath product
of the symmetric groups $S_m$ and $S_n$, acting as a transitive imprimitive
subgroup of $S_{mn}$. Let $\mu$ and $\nu$ be partitions of $m$ and~$n$, respectively.
Let $S^\lambda$ denote the Specht module
for $\Q S_r$ labelled by the partition $\lambda$ of a natural number~$r$.
The object of this article is to prove Theorems~\ref{thm:mainMaximal} and~\ref{thm:main} below
which give a combinatorial
characterization of the maximal and
minimal partitions~$\lambda$ in the dominance order on partitions of $mn$
such that $S^\lambda$ is a summand of the \emph{generalized Foulkes module}
\[ H_\mu^\nu = (S^\mu \oslash S^\nu) \indmn. \]
Here $S^\mu \oslash S^\nu$ denotes the $\Q (S_m \wr S_n)$-module with underlying vector
space $(S^{\mu})^{\otimes n} \otimes S^\nu$ defined in \S\ref{subsec:oslash}.
Two notable features of the proof are that it is carried out entirely in the symmetric
group, and that it gives an explicitly defined non-zero $\Q S_{mn}$-homomorphism
$S^\lambda \rightarrow H_\mu^\nu$ for each maximal or minimal partition~$\lambda$.
The required background is relatively light: a reader familiar with the basic results
on symmetric group representations in \cite{James} should find the proof is self-contained.

Restated in the language of polynomial representations of infinite general linear groups,
Theorems~\ref{thm:mainMaximal} and~\ref{thm:main} determine the maximal and minimal partitions
labelling the irreducible summands $\nabla^\lambda(E)$ of $\nabla^\nu \bigl( \nabla^\mu E \bigr)$, where
$\nabla^\lambda$ is the Schur functor for the partition $\lambda$ and $E$ is a 
$\Q$-vector space of sufficiently
high dimension.
In particular, the special cases $\nu = (n)$ or $\nu = (1^n)$ give
new results on the summands of $\Sym^n \bigl( \nabla^\mu E)$ and $\bigwedge^n \bigl(
\nabla^\mu E)$.
Equivalently, Theorems~\ref{thm:mainMaximal} and~\ref{thm:main} 
determine the maximal and minimal partitions
$\lambda$ such that~$s_\lambda$ appears with a non-zero coefficient in the plethysm
$s_\nu \circ s_\mu$ of the Schur functions $s_\nu$ and~$s_\mu$. 
We discuss these connections
and give further applications of our main theorem in \S\ref{sec:applications} below.
In particular, we prove three
conjectures of Agaoka \cite[Conjectures~1.2, 2.1 and 4.1]{Agaoka} 
on the lexicographically greatest and least constituents of $s_\nu \circ s_\mu$.
We survey the relatively few existing results on these plethysms, and give
some further motivation for the study of plethysms, later in the introduction.

The following combinatorial definitions are needed to state our main theorems.
Example~\ref{ex:ex21} below illustrates these definitions and theorems.

\begin{definition}{\ }\label{defn:families}
\begin{defnlist}
\item A tableau with entries in $\N$ is \emph{conjugate-semistandard} if its rows are 
strictly increasing and its columns are weakly increasing.

\item A \emph{conjugate-semistandard tableau family} of \emph{shape} $\mu^d$ is
a set of $d$  conjugate-semistandard $\mu$-tableaux.

\item Let $\kappa$ be a partition of $n$ with exactly $c$ parts.
A \emph{conjugate-semistandard tableau family tuple}
of \emph{shape $\mu^\kappa$} is a tuple $(\mathcal{T}_1,\ldots, \mathcal{T}_c)$
where $\mathcal{T}_i$ is a conjugate-semistandard tableau family of shape $\mu^{\kappa_i}$,
for each $i$.

\item 
Let $(\mathcal{T}_1,\ldots, \mathcal{T}_c)$ be a conjugate-semistandard
tableau family tuple. Let $M$ be the greatest entry of the $\mu$-tableaux
in the families $\mathcal{T}_1,\ldots,\mathcal{T}_c$.
 The \emph{weight} of $(\mathcal{T}_1,\ldots, \mathcal{T}_c)$
is the composition $(\gamma_1,\ldots,\gamma_M)$
such that the total number of occurrences of $j$
in the $\mu$-tableaux in the families $\mathcal{T}_1,\ldots,\mathcal{T}_c$ is
$\gamma_j$, for each $j \in \{1,\ldots, M\}$.

\item The \emph{type} of a conjugate-semistandard tableau family tuple of
weight~$\gamma$, where $\gamma$ is a partition, is the conjugate partition $\gamma'$.

\end{defnlist}
\end{definition}

Our main theorems are as follows.

\begin{theorem}\label{thm:mainMaximal}
Let $\mu$ be a partition of $m$ and let $\nu$ be a partition of $n$.
The maximal partitions $\lambda$ in the dominance order
that label 
the Specht modules~$S^\lambda$ such
that $S^\lambda$ is a summand of $H_\mu^\nu$ are
precisely the maximal partitions that are weights of  conjugate-semistandard
tableau family tuples of shape $\eta^{\nu'}\!$ where $\eta = \mu'$. 
\end{theorem}

\begin{theorem}\label{thm:main}
Let $\mu$ be a partition of $m$ and let $\nu$ be a partition of $n$. 
Set $\kappa = \nu'$ if $m$ is even, and $\kappa = \nu$ if $m$ is odd.
The minimal partitions $\lambda$ in the dominance order
that label the 
Specht modules $S^\lambda$
such that $S^\lambda$ is a summand of $H_\mu^\nu$
are precisely the minimal types of the
conjugate-semistandard tableau family
tuples of shape $\mu^\kappa$.
\end{theorem}

We show in \S\ref{sec:proof} that
Theorem~\ref{thm:mainMaximal} follows easily from Theorem~\ref{thm:main} by tensoring
with the sign representation of $S_{mn}$. We therefore focus
on the proof of Theorem~\ref{thm:main}.
Throughout this article the partitions
$\lambda$, $\mu$, $\nu$ and $\kappa$ have the meanings 
in Theorems~\ref{thm:mainMaximal} and~\ref{thm:main}. 

The special cases of Theorem~\ref{thm:main} when $\mu = (m)$ or $\mu = (1^m)$ were 
proved by the authors in \cite{PagetWildonTwisted}; in these cases a conjugate-semistandard
tableau family of shape $\mu^d$
is simply a family of~$d$ distinct subsets or multisubsets
(respectively) of $\N$.
It appears to be impossible to make a routine generalization of the proof in \cite{PagetWildonTwisted}, 
and so we have at many points adopted a different approach.
In particular, we do not use Garnir relations to prove that
our homomorphisms from Specht modules to generalized Foulkes modules are
well-defined, instead relying on the combinatorial argument 
given in \S\ref{sec:proof}. 

Our methods also bound the multiplicity of the Specht module
summands given by Theorems~\ref{thm:mainMaximal} and~\ref{thm:main}: the multiplicity
is bounded above in Theorem~\ref{thm:toFamily} and below in Proposition~\ref{prop:linindep}.
In Corollary~\ref{cor:minMult} we show that the multiplicity of the Specht module summand of $H_\mu^\nu$
labelled by the lexicographically least partition may be arbitrarily large;
by contrast, the multiplicity of the Specht module labelled by the lexicographically
greatest partition is always~$1$, by Corollary~\ref{cor:lexMax}.

In the following example, and in some later arguments,
it is useful
to extend the definition of weights and types to conjugate-semistandard tableau families.
We do this
in the obvious way, by regarding such a family as a conjugate-semistandard tableau family of shape
$\mu^{(d)}$.

\begin{example}\label{ex:ex21}
Part of the poset of conjugate-semistandard $(2,1)$-tableaux under
the majorization order $\preceq_\maj$  in Definition~\ref{defn:majorizationTableaux} 
is shown in Figure~1 overleaf.
By Lemma~\ref{lemma:minimalimpliesclosed}, if $(\mathcal{T}_1,\ldots,\mathcal{T}_c)$ 
is a conjugate-semistandard tableau family tuple of minimal type then
each $\mathcal{T}_i$ is downwardly closed under $\preceq_\maj$; that is, if $t \in \mathcal{T}_i$
and $s \preceq_\maj t$, then $s \in \mathcal{T}_i$. 
The five closed conjugate-semistandard tableau families of shape $(2,1)^4$ 
 are shown in Figure~2. The first four  have minimal type; note the third and fourth share the same minimal type.
Therefore, by Theorem~\ref{thm:main}, the minimal partitions $\lambda$ 
such that 
$S^\lambda$ is a summand of the generalized Foulkes module $\smash{H_{(2,1)}^{(4)}}$
are $(5,1^7)$, $(4,2,2,1^4)$ and $(3,3,2,2,1,1)$. Proposition~\ref{prop:linindep} implies
that the multiplicity of $S^{(3,3,2,2,1,1)}$ is~$2$. 
Since $(8,4)$ is the
weight of the unique closed conjugate-semistandard tableau family tuple of shape $(2,1)^{(1,1,1,1)}$,
it follows from Theorem~\ref{thm:mainMaximal} that
the unique maximal partition~$\lambda$ such that~$S^\lambda$ is a summand
of $H_{(2,1)}^{(4)}$ is $(8,4)$. 

We leave it as an exercise to the reader to use Theorems~\ref{thm:mainMaximal} and~\ref{thm:main} 
together with Figure~1 and Lemma~\ref{lemma:minimalimpliesclosed}
to show that the minimal partitions~$\lambda$ such that $S^\lambda$ is a summand of $H_{(2,1)}^{(3,1)}$ 
are $(4,2,1^6)$ and $(3,2,2,2,1^3)$, and the maximal partitions are $(8,3,1)$ and $(7,5)$.
\end{example}

\begin{figure}[h]
\begin{tikzpicture}[scale=1, x=1.25cm, y=1.5cm]

\draw (0.4,0.4)--(0.6,0.6);
\draw[xshift=1.25cm,yshift=1.5cm] (0.4,0.4)--(0.6,0.6);
\draw[xshift=2.5cm,yshift=3cm] (0.4,0.4)--(0.6,0.6);
\draw[xshift=-1.25cm,yshift=1.5cm] (0.4,0.4)--(0.6,0.6);
\draw[xshift=0cm,yshift=3cm] (0.4,0.4)--(0.6,0.6);
\draw[xshift=-2.5cm,yshift=3cm] (0.4,0.4)--(0.6,0.6);

\draw (-0.4,0.4)--(-0.8,0.8);
\draw[xshift=1.25cm,yshift=1.5cm] (-0.4,0.4)--(-0.8,0.8);
\draw[xshift=2.5cm,yshift=3cm]  (-0.4,0.4)--(-0.8,0.8);
\draw[xshift=-1.25cm,yshift=1.5cm]  (-0.4,0.4)--(-0.8,0.8);
\draw[xshift=0cm,yshift=3cm]  (-0.4,0.4)--(-0.8,0.8);
\draw[xshift=-2.5cm,yshift=3cm]  (-0.4,0.4)--(-0.8,0.8);

\draw (0,2.4)--(0,2.6);

\node at (0,0) [] {$\young(12,1)$};
\node at (1,1) [] {$\young(13,1)$};
\node at (-1,1) [] {$\young(12,2)$};
\node at (-2,2) [] {$\young(12,3)$};
\node at (0,2) [] {$\young(13,2)$};
\node at (2,2) [] {$\young(14,1)$};
\node at (-3,3) [] {$\young(12,4)$};
\node at (-1,3) [] {$\young(13,3)$};
\node at (0,3) [] {$\young(23,2)$};
\node at (1,3) [] {$\young(14,2)$};
\node at (3,3) [] {$\young(15,1)$};
\end{tikzpicture}
\caption{Part of the poset of conjugate-semistandard $(2,1)$-tableaux under the majorization order.}
\end{figure}
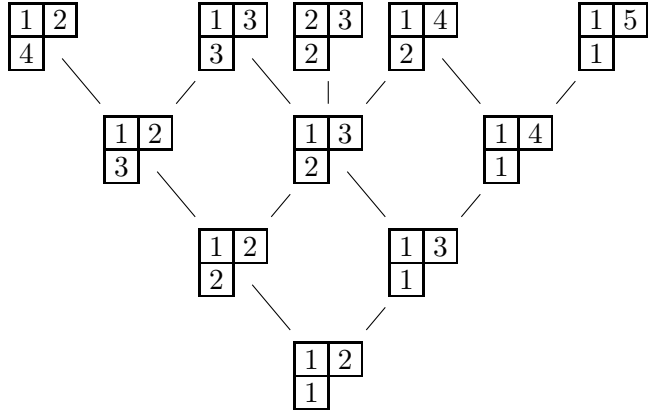

\begin{figure}[h]
\begin{center}
\begin{tabular}{ll} \\ \Xhline{2\arrayrulewidth}
Tableau family & Type\rule[-6pt]{0pt}{19pt} \\ \Xhline{2\arrayrulewidth}
$\Bigl\{\,\young(12,1)\sp\young(13,1)\sp\young(14,1)\sp \young(15,1)\,\Bigr\}$\hs  & $(5,1^7)$ \rule[-15pt]{0pt}{34pt} \\
$\Bigl\{\,\young(12,1)\sp\young(12,2)\sp\young(13,1)\sp\young(14,1)\,\Bigr\}$\hs & $(4,2,2,1^4)$\rule[-15pt]{0pt}{33pt} \\
$\Bigl\{\,\young(12,1)\sp\young(12,2)\sp\young(13,1)\sp \young(13,2)\,\Bigr\}$\hs & $(3,3,2,2,1,1)$\rule[-15pt]{0pt}{33pt} \\
$\Bigl\{\,\young(12,1)\sp\young(12,2)\sp \young(12,3)\sp\young(13,1)\,\Bigr\}$\hs & $(3,3,2,2,1,1)$ \rule[-15pt]{0pt}{33pt}\\
$\Bigl\{\,\young(12,1)\sp\young(12,2)\sp\young(12,3)\sp \young(12,4)\,\Bigr\}$\hs & $(4,2,2,2,2)$ \rule[-15pt]{0pt}{33pt} \\
\Xhline{2\arrayrulewidth}
\end{tabular}
\end{center}
\caption{The closed conjugate-semistandard tableau families of shape $(2,1)^4$.}
\end{figure}


\subsection*{Motivation}
D.~E.~Littlewood defined the plethysm of symmetric functions in 1936.  
Since then  progress on the decomposition of the plethysm $s_\nu \circ s_\mu$ as an integral linear combination of Schur functions has been made via symmetric functions, via polynomial representations of general linear groups and via the representation theory of symmetric groups. However, few general results have been found. 
Finding a combinatorial rule for 
the multiplicities $\langle s_\nu \circ s_\mu, s_\lambda\rangle$ that demonstrates their positivity 
was identified as a key open problem in algebraic combinatorics by Stanley in~\cite[Problem 9]{StanleyPositivity}.

One motivation for this problem is a long-standing conjecture of Foulkes
\cite{Foulkes}, which states that if $m \le n$ then 
$s_{(n)} \circ s_{(m)} - s_{(m)} \circ s_{(n)}$ is a non-negative integral linear
combination of Schur functions. Foulkes' Conjecture 
was proved when $m=2$ in \cite{Thrall}, when $m=3$ in \cite{DentSiemons}, when $m=4$
in \cite{McKay}, and recently when $m=5$ in \cite{CheungIkenmeyerMkrtchyan}.
It has also been proved when $n$ is very much larger than $m$ by Brion \cite{Brion}, using
methods from geometric invariant theory. In this connection, we note that
if $E$ is a complex vector
space, then the decomposition
of $s_{(n)} \circ s_\mu$ gives the polynomial representations of $\GL(E)$ appearing
in the $n$th component $\Sym^n \nabla^\mu(E)$ of the coordinate ring $\mathcal{O}(\nabla^\mu(E))$.
Our results apply to these rings;  we find some new invariants for the
special linear group at the end of \S\ref{sec:applications}.

The plethysms $\nabla^\lambda\!\bigwedge^m(E)$ appear in the problem of describing
relations between the $m\times m$ minors of generic matrices \cite[\S 1.5]{BrunsConcaVarbaro}.
Plethysms can also be used to construct polynomial representations of general linear groups: for example,
$\nabla^{(2,2)}(E)$ is the kernel of the canonical surjection from $\Sym^2(\Sym^2(E))$
to $\Sym^{(4)}(E)$; thus, $s_2 \circ s_2 = s_{(2,2)} + s_{(4)}$. This
is generalized and given a geometric interpretation in \cite[\S 14.4]{FultonHarris}.
A further motivation comes from enumerative combinatorics. For example, by \cite[(5.8)]{Read},
the number of $k$-regular graphs on~$n$ vertices (counted up to isomorphism)
is $\langle s_{(n)} \circ s_{(k)}, s_{(nk/2)} \circ s_2 \rangle$. Finally we mention
that a better understanding of plethysm coefficients is essential to Mulmuley and Sohoni's geometric approach
to the VP $\not=$ VNP problem in algebraic complexity theory: see \cite{BurgisserEtAl} for an introduction.

\subsection*{Background}
The existing results on the plethysm of Schur functions fall into three classes: explicit decompositions of 
$s_\nu \circ s_\mu$ for particular $\nu$ or $\mu$, theorems on constituents of special forms, and results which relate plethysm coefficients. Most of these results were obtained by symmetric function methods; for uniformity, we use this language throughout our survey.

Explicit decompositions of the plethysms $s_\nu \circ s_{(m)}$ are known when $\nu$ is a partition of
$n \le 4$; see Littlewood \cite{LittlewoodCharacters} for $n=2$, Thrall \cite[Theorem~5]{Thrall} or Dent and Siemons \cite[Theorem~4.1]{DentSiemons} for $\nu=(3)$, and Foulkes \cite[Theorem 27]{FoulkesPlethysm} and Howe \cite[Section 3.5 and Remark 3.6(b)]{Howe} for the remaining cases.  For sufficiently small partitions $\nu$ and $\mu$, the plethysm $s_\nu \circ s_{\mu}$  can readily be calculated using any of the computer algebra systems {\sc Magma} \cite{Magma}, {\sc Gap}~\cite{GAP4} or {\sc Symmetrica} \cite{Symmetrica}. (There are 
also many hand calculations in early papers: see for example \cite[pages 383, 388]{Thrall}.)
A new recurrence satisfied by the multiplicities
$\langle s_{(n)} \circ s_{(m)}, s_\lambda \rangle$ was given in \cite[Proposition 5.1]{EvseevPagetWildon} and used to verify Foulkes' Conjecture 
for $m+n \le 19$, extending an earlier computation in \cite{MuellerNeunhoffer} for $m+n \le 17$.
This recurrence was generalized to the plethysms $s_\nu \circ s_{(m)}$ in 
\cite[Theorem 6.2.6]{deBoeckThesis}. The plethysms $s_2 \circ s_\mu$ and $s_{(1^2)} \circ s_\mu$,
describing the decomposition of $\nabla^\mu(E) \otimes \nabla^\mu(E)$ into its symmetric and antisymmetric
parts, are determined by \cite[Theorem~4.1]{CarreLeclerc}.

Several other results, like our Theorems~\ref{thm:mainMaximal} and~\ref{thm:main}, give information 
about constituents of a special form.  
By the Cayley--Sylvester formula, the multiplicity of $s_{(mn-d,d)}$ in $s_{(n)} \circ s_{(m)}$ 
is the number of partitions of $d$ whose Young diagram is contained in the Young diagram of $(m^n)$,
minus the number of partitions of $d-1$ satisfying the same restriction.
The multiplicities of the Schur functions 
$s_{(mn-d,1^d)}$, $s_{(mn-d-s,s,1^{d})}$ and $s_{(mn-d-2t,2^t,1^d)}$,
in $s_\nu \circ s_\mu$ were found by Langley and Remmel in \cite{LangleyRemmel}.
Giannelli \cite[Theorem~1.2]{GiannelliArchMath} later found  the multiplicities of a wider
class of  constituents of $s_{(n)} \circ s_{(m)}$ labelled by `near-hook' partitions.

Foulkes' Conjecture gives one conjectural relationship between plethysm coefficients. 
There are further theorems which given a constituent of a plethysm of Schur functions yield constituents of related plethysms, such as Foulkes' Second Conjecture, proved (in a more general form)
by Brion in \cite[\S 2]{Brion}, which
states that 
$\langle s_{(n)} \circ s_{(m)}, s_\lambda \rangle \le \langle s_{(n)} \circ s_{(m+1)}, s_{\lambda + (n)}
\rangle$ for any partition~$\lambda$ of $mn$; here the addition is componentwise, 
as for weights of representations of general
linear groups. This setting may also be used to prove Proposition~4.3.4 in 
\cite{IkenmeyerThesis}, which states that if $\langle s_n \circ s_m, s_\lambda \rangle \ge 1$
and $\langle s_{\tilde{n}} \circ s_m, s_{\tilde{\lambda}} \rangle \ge 1$ then
$\langle s_{n + \tilde{n}} \circ s_m, s_{\lambda + \tilde{\lambda}} \rangle \ge 1$.
Newell proved in \cite{Newell} that 
$\langle s_{(n)} \circ s_{(m)}, s_\lambda \rangle = \langle s_{(1^n)} \circ s_{(m+1)}, s_{\lambda + (1^n)}
\rangle$ and $\langle s_{(1^n)} \circ s_{(m)}, s_\lambda \rangle = \langle s_{(n)} \circ s_{(m+1)},
s_{\lambda+{(1^n)}} \rangle$.
A result of a similar flavour was found by de Boeck in \cite{deBoeckThesis} who showed that 
$\langle s_{(n)} \circ s_{(1^m)}, s_\lambda \rangle \le \langle s_{(n+1)} \circ s_{(1^m)}, s_{\lambda +(1^m)}
\rangle$.
Another relationship between plethysms was proved by 
 Manivel \cite{ManivelCayleySylvester}, who showed  that the specialization
 $(s_{(n^k)} \circ s_{(m+k-1)})(q^{-1},q)$  is symmetric under any permutation of $m$, $n$ and $k$,  
 generalizing  the Cayley--Sylvester formula. Taking $k=1$ and swapping $k$ and~$n$ gives
 $(s_{(n)} \circ s_{(m)})(q^{-1},q) = (s_{(1^n)} \circ s_{(m+n-1)})(q^{-1},q)$; this is a
 combinatorial expression of the Wronskian isomorphism: see for example \cite[\S 2.5]{AbdesselamChipalkatti}.

\subsection*{Outline}\label{subsec:outline} 
In \S 2 we set out our notation and prove some preliminary results.
In \S\ref{sec:models} we define explicit models
for the modules $H_\mu^\nu$ using
$\nu$-tableaux whose entries are certain $\mu$-tableaux. We believe these models
are of independent interest and will be of use beyond their applications in this paper. 
In \S\ref{sec:homs} we define,
for each  conjugate-semistandard tableau family tuple of 
shape $\mu^\kappa$ and type $\lambda$,
a homomorphism from $\widetilde{M}^\lambda$ into our model for $H_\mu^\nu$,
where $\widetilde{M}^\lambda$ is the twisted Young permutation module defined in \S\ref{subsec:Specht}
below.
In \S\ref{sec:nonzero} we show that these homomorphisms are non-zero. In \S\ref{sec:toTuple}
we show that given a non-zero homomorphism
$S^\lambda \rightarrow H_\mu^\nu$, 
there is a conjugate-semistandard tableau family tuple of shape $\mu^\kappa$.
These two sections contain the main work in the proof, and may be read independently of each other.
In \S\ref{sec:proof} we combine them to prove Theorem~\ref{thm:main};
we then deduce Theorem~\ref{thm:mainMaximal} by an easy sign twist.

The outline
of the proof of Theorem~\ref{thm:main} is as follows: by Proposition~\ref{prop:nonzero}, if there is a 
conjugate-semistandard tableau family tuple of type $\lambda$ and shape~$\mu^\kappa$,
then there is a non-zero homomorphism $\widetilde{M}^\lambda \rightarrow H_\mu^\nu$.
 By Lemma~\ref{lemma:Mtilde}, it follows that there is a summand $S^{\lambda^\dagger}$
of $H_\mu^\nu$ for some $\lambda^\dagger \unlhd \lambda$. (Here $\unlhd$ denotes the dominance order on partitions.) On the other hand,
by Theorem~\ref{thm:toFamily}, given any such summand, there is a conjugate-semistandard
tableau family tuple of the corresponding type. Hence tuples of minimal type
correspond to the summands of $H_\mu^\nu$ labelled by minimal partitions.

In \S\ref{sec:indep} we give a sufficient condition for the homomorphisms
defined with respect to distinct conjugate-semistandard tableau family tuples of the same type
to be linearly independent. Example~\ref{ex:big} gives an indication of the more complicated
behaviour when this condition does not hold.   We end in \S\ref{sec:applications}
with applications of this result and our
two main theorems  to the conjectures of Agaoka mentioned earlier.
In addition, we characterize all generalized Foulkes modules having a unique
minimal Specht module summand in the dominance order.

\section{Preliminary definitions and results}
\label{sec:preliminaries}

\subsection{Young diagrams}
We define the \emph{Young diagram} of a partition $\gamma$ with
exactly $k$ parts
by $[\gamma] = \{(i,j) : 1 \le i \le k, 1 \le j \le \gamma_i \}$. We
refer to the elements of $[\gamma]$ as \emph{boxes} and draw Young diagrams
in the `English' convention, where the largest part appears at the top.

\subsection{Tableaux}\label{subsec:tableau}

Let $\gamma$ be a partition and let $\Omega$ be a set. A \emph{$\gamma$-tableau $t$ 
with entries from $\Omega$}
is a function $t : [\gamma] \rightarrow \Omega$. 
If $(i,j) \in [\gamma]$ and $t$ is a tableau with entries from $\Omega$ 
then the \emph{entry in position} $(i,j)$ of $t$
is $(i,j)t \in \Omega$. 
If $t$ is injective then we say the
entries are \emph{distinct}. 
The \emph{conjugate} of a $\gamma$-tableau~$t$ 
is the $\gamma'$-tableau defined by $(i,j)t' = (j,i)t$.
The symmetric group $S_\Omega$ acts on the set of $\gamma$-tableaux with entries from $\Omega$:
if $\phi\in S_\Omega$ and $t$ is such a tableau,
then $t \phi$ is the $\gamma$-tableau defined by $(i,j)(t \phi) = \bigl( (i,j)t\bigr) \phi$. 

We also need a place permutation action: if $\sigma \in S_{[\gamma]}$, 
the symmetric group on the boxes of the Young diagram $[\gamma]$,
 then $t \pp \sigma$
is the $\gamma$-tableau defined by $(i,j)(t \pp \sigma) = \bigl( (i,j) \sigma^{-1} \bigr) t$.
Thus if $t$ has entry $\alpha$ in position $(i,j)$ then $t \pp \sigma$ has entry $\alpha$ in
position $(i,j)\sigma$. We define the groups of \emph{row place permutations of $\gamma$}
and \emph{column place permutations of~$\gamma$}
to be the Young subgroups $\RPP(\gamma)$ and $\CPP(\gamma)$ 
of $S_{[\gamma]}$
having as their orbits the rows of $[\gamma]$
and the columns of $[\gamma]$, respectively.
Given a $\gamma$-tableau $t$ with entries from $\Omega$, 
we define the \emph{tabloid} $\{t\}$ and \emph{columnar tabloid} $||t||$ 
to be the equivalence classes of $t$
up to the action of $\RPP(\gamma)$ and $\CPP(\gamma)$, respectively.
Thus $\{t\}$ and $||t||$ are determined by
the multisets of entries in the rows and columns of $t$, respectively.

We represent $\gamma$-tableaux with entries from $\Omega$ by
filling the boxes of $[\gamma]$ with elements of~$\Omega$; if $t$ is a $\gamma$-tableau
then we represent the corresponding tabloid~$\{t\}$ by removing vertical
lines from $t$ and the corresponding columnar tabloid~$||t||$ by removing horizontal
lines from $t$.

For example, if $t = \young(12,13)$ then
$\{t\} = \yngt(12,13)\hskip1pt$ and $||t|| = \youngcolumntabloid(12,13)\hskip1pt$;
here $\gamma=(2,2)$ and $\Omega$ may be taken to be $\N$. 

\subsection{Specht and dual Specht modules}
\label{subsec:Specht}

Fix $r \in \N$, a partition $\gamma$ of~$r$ and a set $\Omega$ of size $r$.
Throughout this subsection all tableaux, and so all tabloids and columnar tabloids,
have distinct entries from $\Omega$.

The action of $S_\Omega$ on $\gamma$-tableaux with entries from $\Omega$ 
commutes with the place permutation
action of $S_{[\gamma]}$. 
Thus setting $\{t\}\phi = \{t \phi\}$ 
for $\phi \in S_\Omega$ and~$\{t\}$ a $\gamma$-tabloid
gives a well-defined action of $S_\Omega$ on the set of $\gamma$-tabloids.
Let $M^\gamma$ denote the \emph{Young permutation
module}  for $\Q S_\Omega$ spanned by all $\gamma$-tabloids.
Let $\widetilde{M}^\gamma = M^{\gamma'} \otimes \langle v\rangle$ where $\langle v \rangle$
affords the sign representation of $S_\Omega$. 
Fix a $\gamma$-tableau $t_\gamma$ with entries from $\Omega$, so that $\{t_\gamma\} \otimes v$ generates
$\widetilde{M}^\gamma$. We write $|t_\gamma|=\{t_\gamma'\} \otimes v$ and $|t_\gamma \phi|=(\{t_\gamma'\} \otimes v)\phi \in \widetilde{M}^\gamma$ for $\phi \in S_\Omega$.
%
Thus if $t$ is a $\gamma$-tableau and $\tau \in \CPP(\gamma)$ then  $|t \pp \tau | = |t|\sgn(\tau)$;
  up to a sign, $|t|$ is determined by the corresponding
 columnar tabloid $||t||$.
 We note that our definition of $\widetilde{M}^\gamma$ 
agrees with Fulton's in \cite[Chapter~7]{FultonYT}.

Given a $\gamma$-tableau $t$, let $e(t) \in M^\gamma$ be the corresponding 
\emph{polytabloid}
defined by $e(t) = \sum_{\tau \in \CPP(\gamma)} \{t \pp \tau \} \sgn(\tau)$
and let $\widetilde{e}(t) \in \widetilde{M}^\gamma$ be the corresponding 
\emph{dual polytabloid}
defined by $\widetilde{e}(t) = \sum_{\sigma \in \RPP(\gamma)} |t \pp \sigma|$.
Let $S^\gamma \subseteq M^\gamma$ be the \emph{Specht module}
spanned by all 
polytabloids $e(t)$ for $t$ a $\gamma$-tableau.
Let $\smash{\widetilde{S}^\gamma \subseteq \widetilde{M}^\gamma}$ be the \emph{dual Specht module}
spanned by all 
dual polytabloids $\widetilde{e}(t)$ for $t$ a $\gamma$-tableau. 
Since the action of $S_\Omega$ commutes with the place permutation action of $\CPP(\gamma)$, we have
\begin{equation}\label{eq:SpechtCyclic} e(t) \phi = e(t\phi) \end{equation}
for  $\phi \in S_\Omega$. Hence the Specht module
$S^\gamma$ is cyclic, generated by any 
polytabloid. (An analogous result
holds for $\widetilde{S}^\gamma$.)

%

There are canonical $\Q S_\Omega$-homomorphisms $\widetilde{M}^\gamma \rightarrow S^\gamma$
and $M^\gamma \rightarrow \widetilde{S}^\gamma$ defined by $|t| \mapsto e(t)$ and
$\{t\} \mapsto \widetilde{e}(t)$, respectively. By \cite[\S 7.4, Lemma 5]{FultonYT},
the composition $\rule{0pt}{11pt}S^\gamma \rightarrow M^\gamma \rightarrow \widetilde{S}^\gamma$ is
an isomorphism.
By \cite[Theorems 4.12 and 6.7]{James}, Specht and dual Specht modules
are irreducible and each irreducible module for $\Q S_\Omega$
is isomorphic to a unique Specht module $S^\delta$ for $\delta$ a partition of $r$.
Moreover,  
\begin{equation}
\label{eq:sgntwist}
S^\lambda \otimes \sgn \cong \widetilde{S}^{\lambda'} \cong S^{\lambda'}.
\end{equation}

For example, 
the Specht module $S^{(2,2)}$ and the dual Specht module $\widetilde{S}^{(2,2)}$ for
$\mathbb{Q}S_4$ are generated by
\begin{align*}
e(\, \young(12,34) \, ) &= \yngt(12,34) - \yngt(23,14) - \yngt(14,23) + \yngt(34,12)\, ,  \\[3pt]
\widetilde{e}(\, \young(12,34)\, ) &= \youngcolumntabloid(12,34)
+\youngcolumntabloid(21,34) +\youngcolumntabloid(12,43) + \youngcolumntabloid(21,43)\, ,
\end{align*}
respectively. Setting $t = \young(12,34)\,\rule[-11pt]{0pt}{12pt}$, these generators are 
the images of  $|t| \in \widetilde{M}^{(2,2)}$ and $\{t\} \in M^{(2,2)}$, 
under the respective canonical homomorphisms.

As mentioned in the outline, the following  lemma is key to the proof of the main theorem.

\begin{lemma}\label{lemma:Mtilde}

If $S^{\lambda^\dagger}$ is a summand of $\widetilde{M}^\lambda$,
for some partition ${\lambda^\dagger}$ of $r$,
 then $\lambda \unrhd \lambda^\dagger$.
\end{lemma}

\begin{proof}
Multiplying by the sign representation using~\eqref{eq:sgntwist}, this follows from Young's rule (see, for instance
\cite[Corollary 1, page 92]{FultonYT}) or \cite[Proposition~7.18.7]{StanleyII}).
\end{proof}

\subsection{Symbols}\label{subsec:symbols}
Let $\lambda$ be a partition.
We define the \emph{symbol} of a box $(x,y) \in [\lambda]$ to be the formal
symbol~$y_x$. We say that $y$ is the \emph{number} and $x$ is the \emph{index} of $y_x$.
Let $\Omega^\lambda$ be the set of all symbols $y_x$ for $(x,y) \in [\lambda]$.
Let $I(\lambda) \le S_{\Omega^\lambda}$ be the Young subgroup of $S_{\Omega^\lambda}$
having as its orbits the sets $\{y_x : 1 \le x \le \lambda_y'\}$ for each $y \in \{1,\ldots,\lambda_1\}$.
We order $\Omega^\lambda$ by first comparing on numbers,
then on indices: thus $y'_{x'} < y_x$ if and only if $y' < y$ or $y'=y$ and $x' < x$.
Let $t_\lambda$ be the $\lambda$-tableau with entries from~$\Omega^\lambda$ defined by
$(x,y)t_\lambda = y_x$. Since $I(\lambda)$ permutes entries within the columns of $t_\lambda$,
we have $e(t_\lambda) = \{t \} b_\lambda$ where
$b_\lambda = \sum_{\theta \in I(\lambda)} \theta \sgn(\theta)$.

\subsection{Total order on conjugate-semistandard tableaux}
\label{subsec:totalOrder}

Let $\Omega$ be totally ordered by $<$ and let
$A = \{\alpha_1,\ldots, \alpha_d\}$ and $B = \{\beta_1,\ldots, \beta_d\}$ be multisubsets of $\Omega$
such that $\alpha_1 \le \ldots \le \alpha_d$ and $\beta_1 \le \ldots \le \beta_d$. 
The \emph{colexicographic order} on multisubsets of~$\Omega$
is defined by $A < B$ if and only
if for some~$q$ we have $\alpha_q < \beta_q$ and $\alpha_{q+1} = \beta_{q+1}, \ldots, \alpha_d = \beta_d$.
It is a total order.

\begin{definition}\label{defn:colorder}
Let $u$ and $v$ be distinct conjugate-semistandard $\mu$-tableaux with entries in $\N$.
Let column $j$ be the rightmost column in which $u$ and $v$ differ,
and let $A$ and~$B$ be the multisets of entries in column $j$ of $u$ and $v$, respectively.
We set $u < v$ if and only if $A < B$ in the colexicographic order on multisubsets of~$\N$.
\end{definition}


\subsection{Pre-orders on tableaux}

It will be useful to compare tableaux under relations that are reflexive and 
transitive, but fail (in general) to be antisymmetric. Such relations are 
called pre-orders.
Let $\preceq$ be a pre-order on a set $\Omega$. Thus $\alpha \preceq \alpha$ 
and if $\alpha \preceq \beta$ and $\beta \preceq 
\gamma$ then $\alpha \preceq \gamma$ for all $\alpha ,\beta,\gamma \in \Omega$. 
There may exist 
distinct $\alpha$ and $\beta$ such that $\alpha \preceq \beta$ and 
$\beta \preceq \alpha$.  
We write $\alpha \prec \beta$ to mean that $\alpha\preceq \beta$ and $\beta\npreceq \alpha$. 
A key property we need is stated in the following lemma.

\begin{lemma}\label{lemma:alphabeta}
Let $\preceq$ be a pre-order on a set $\Omega$ and let $\alpha_1, \ldots, \alpha_d$,
$\beta_1,\ldots,\beta_d \in \Omega$. If $\alpha_i \preceq \beta_i$ for each $i$,
and there exists a permutation $\phi \in S_d$ such that 
$\beta_{i\phi} \preceq \alpha_i$ for each $i$,
then $\beta_i \preceq \alpha_i$ for each $i$.
\end{lemma}

\begin{proof}
Let $i \in \{1,\ldots, d\}$. If $i$ is in a cycle of $\phi$ of length $r$ then
\[ \alpha_i \succeq \beta_{i\phi} \succeq \alpha_{i\phi} \succeq \ldots \succeq \beta_{i \phi^{r-1}}
\succeq \alpha_{i \phi^{r-1}} \succeq \beta_{i \phi^r} = \beta_i. \]
By transitivity, it follows that $\beta_i \preceq \alpha_i$, as required.
\end{proof}


\begin{definition}\label{defn:preorder}
Let $\preceq$ be a pre-order on a set $\Omega$.
Let $u$ and $v$ be  tableaux of the same shape
with entries from $\Omega$. 

\begin{defnlist}
\item We set that $u \preceq_\col v$ if either \emph{(a)}  the multisets of entries of $u$ and $v$ agree in every column, or \emph{(b)}
it is possible to order the entries   in the rightmost column for which the multisets differ  so that $u$ has entries $\alpha_1, \ldots, \alpha_d$
and $v$ has entries $\beta_1,\ldots,\beta_d$ with $\alpha_1 \preceq \beta_1$, \ldots, $\alpha_d \preceq
\beta_d$. 
%
\item We set $u \preceq_\row v$ if $u' \preceq_\col v'$.
\end{defnlist}
\end{definition}

It is easily seen that $\preceq_\col$ and $\preceq_\row$ 
are pre-orders on the set of $\gamma$-tableaux with entries from $\Omega$.
Moreover $\preceq_\col$ and $\preceq_\row$ induce pre-orders on the set of $\gamma$-columnar tabloids
and $\gamma$-tabloids with entries from $\Omega$, respectively.
We reuse the symbols $\preceq_\col$ and $\preceq_\row$ for these induced pre-orders.
(An example is given at the end of this subsection.)

\begin{definition}\label{defn:rowcolstd}
Let $\preceq$ be a pre-order on a set $\Omega$. Let $u$ be a tableau with entries from~$\Omega$.
We say that $u$ is \emph{row-standard} 
if whenever $u$ has a row with entries $\alpha_1, \ldots, \alpha_d$,
read from left-to-right, we have $\alpha_1 \prec \ldots \prec \alpha_d$.
We say that $u$ is 
\emph{column-standard} 
if $u'$ is row-standard. 
\end{definition}

\begin{lemma}\label{lemma:preorder}
Let $\preceq$ be a pre-order on a set $\Omega$ and let $\gamma$ be a partition.
\begin{thmlist}
\item Let $u$ be a row-standard $\gamma$-tableau with entries from $\Omega$.
If $\sigma \in \RPP(\gamma)$ and $\sigma\not=\id$ then $||u \pp \sigma|| \prec_\col ||u||$.
\item Let $v$ be a column-standard $\gamma$-tableau with entries from $\Omega$.
If $\tau \in \CPP(\gamma)$ and $\tau\not=\id$ then $\{ v \pp \tau \} \prec_\row \{v\}$.
\end{thmlist}
\end{lemma} 

\begin{proof}
For (i), let $j$ be maximal such that there exists $i$ with $(i,j)\sigma \not= (i,j)$. 
Let $R = \{i : (i,j)\sigma\not= (i,j)\}$ and let $i \in R$.
Since $\sigma\in \RPP(\gamma)$
we have $(i,j)(u \pp \sigma) = (i,j')u$ for some $j' < j$. Since $u$ is row-standard 
$(i,j')u \prec (i,j)u$. Therefore $(i,j)(u \pp \sigma) \prec (i,j)u$ for each $i \in R$
and $(i,j)(u\pp \sigma) = (i,j)u$ for each $i\not\in R$. It now follows from the definition
of $\prec_{\col}$ (comparing on column $j$) and Lemma~\ref{lemma:alphabeta}
that $||u \pp \sigma|| \prec_\col ||u||$.
The proof of (ii) is similar.
\end{proof}

The orders defined in Definition~\ref{defn:preorder} are
 useful even when $\preceq$ is a total order. For example,
if $\Omega = \N$ and $\preceq$ is the usual total order on $\N$,
then the pre-order
$\preceq_\col$ on tableaux of shape $(2,2)$ satisfies $||s|| \preceq_\col ||t|| \prec_\col ||u||$ and
$||t|| \preceq_\col ||s||$, where
\[ s = \young(12,21)\,, \quad t = \young(21,12)\,, \quad u = \young(12,12)\, .\]
We use this pre-order in Lemma~\ref{lemma:RPP} and the proof of Proposition~\ref{prop:nonzero};
we also use Lemma~\ref{lemma:preorder} in a case where $\preceq$ is itself a pre-order in this proof.

\subsection{Modules for wreath products}
\label{subsec:oslash}

Let $H$ be a finite group and let $n \in \N$. Let $W$ be a $\Q H$-module
and let $Z$ be a $\Q S_n$-module.
The outer tensor product $W^{\otimes n}$ is a module for $\Q H^n$.
The symmetric group $S_n$ acts on this module 
by $(w_1\otimes \cdots \otimes w_n)\sigma =
w_{1\sigma^{-1}} \otimes \cdots \otimes w_{n\sigma^{-1}}$ for $\sigma \in S_n$. 
Combining the actions of $H^n$ and $S_n$ on $W^{\otimes n}$
we obtain a $\Q(H \wr S_n)$-module $X$, as constructed in \S 4.3 of \cite{JK}.
Let $\Inf_{S_n}^{S_m \wr S_n} Z$ denote the $\Q(S_m \wr S_n)$-module inflated from $Z$
using the canonical epimorphism $S_m \wr S_n \rightarrow S_n$.
We define
$W \oslash Z$ to be the $\Q(H \wr S_n)$-module $X \otimes \Inf_{S_n}^{S_m \wr S_n} Z$.
(The symbol $\oslash$ was introduced, in a more general context, in \cite{ChuangTan}.)

Importantly this construction is functorial in both $W$ and $Z$.
Thus given a $\Q H$-homomorphism $f : W \rightarrow W'$ and a $\Q S_n$-homomorphism
$g : Z \rightarrow Z'$ there is a corresponding $\Q (H \wr S_n)$-homomorphism $f \oslash g : 
W \oslash Z \mapsto W' \oslash Z'$, defined by
$\bigl((w_1 \otimes \cdots \otimes w_n) \otimes z\bigr) (f \oslash g) = 
(w_1 f \otimes \cdots \otimes w_n f) \otimes z g$.



\subsection{Closed conjugate-semistandard tableau families}
While not logically essential to the proof of Theorem~\ref{thm:main}, 
in practice it is very useful to know that 
a conjugate-semistandard tableau family tuple of minimal type satisfies the
closure property used in Example~\ref{ex:ex21}. 

\begin{definition}\label{defn:majorizationSets}
Let $A$ and $B$ be finite subsets of a totally ordered set such that
$A = \{\alpha_1,\ldots, \alpha_d\}$ and $B = \{\beta_1,\ldots,\beta_d\}$, where $\alpha_1 < \ldots < \alpha_d$ and
$\beta_1 < \ldots < \beta_d$. If $\alpha_i \le \beta_i$ for each~$i$ then we say that
$B$ \emph{majorizes} $A$.
\end{definition}

\begin{definition}\label{defn:majorizationTableaux}
Let $u$ and $v$ be conjugate-semistandard $\mu$-tableaux with entries in $\N$.
We say that $v$ \emph{majorizes} $u$, and write $u \preceq_\maj v$, if
row $i$ of $v$ majorizes row $i$ of $u$ for each $i$.
\end{definition}

\begin{definition}{\ }\label{defn:closed}
\begin{defnlist}
\item A conjugate-semistandard tableau family $\mathcal{T}$ is \emph{closed}
if, whenever $v\in \mathcal{T}$ and $u$ is a conjugate-semistandard tableau
 such that $u \preceq_\maj v$, we have $u \in \mathcal{T}$.

\item A conjugate-semistandard tableau family tuple $(\mathcal{T}_1,\ldots, \mathcal{T}_c)$
is \emph{closed} if $\mathcal{T}_i$ is closed for each $i$.
\end{defnlist}
\end{definition}

From the
part of the poset of conjugate-semistandard $(2,1)$-tableaux shown in Figure~1, 
one can read off
the five closed conjugate-semistandard tableau families of shape $(2,1)^4$ in Figure~2. 
The first four families listed are of minimal type and the final family is of non-minimal type $(4,2^4)$.
Each of the five closed conjugate-semistandard tableau families has a well-defined
type. This is true in general. 

\begin{proposition}\label{prop:type}
If $\mathcal{T}$ is a closed conjugate-semistandard tableau family then
$\mathcal{T}$ has a well-defined type.
\end{proposition}

The authors have a  proof of Proposition~\ref{prop:type}
by a variation on the Bender--Knuth involution  
(see \cite[page~47]{BenderKnuth}). Since the proposition is used only
in the following result, which is not logically essential to the two main theorems,
and its proof is not short, we have chosen to omit it from this paper.

\begin{lemma}\label{lemma:minimalimpliesclosed}
If  $(\mathcal{T}_1,\ldots, \mathcal{T}_c)$ is a conjugate-semistandard
tableau family tuple of minimal type then  each conjugate-semistandard tableau family $\mathcal{T}_i$ is closed.
\end{lemma}

\begin{proof}
Let $\mathcal{T}_i$ have shape  $\mu^d$. 
Suppose for a contradiction, that $\mathcal{T}_i$
is not closed. 
Then there exists $v \in \mathcal{T}_i$ and a conjugate-semistandard $\mu$-tableau $u$
such that $u \preceq_\maj v$ and $u \not\in \mathcal{T}_i$. 
Choose~$v$ to be minimal in the  $\preceq_\maj$ order for which such a $\mu$-tableau $u$ exists, and then pick ~$u$ maximal in the $\preceq_\maj$ order with the above property.
Let $(i,j)$ be minimal in the lexicographic
order such that $(i,j)u \not= (i,j)v$. Let $(i,j)v = r$. Since $(i,j)u < r$
and $(i-1,j)u = (i-1,j)v$ and $(i,j-1)u=(i,j-1)v$ (when these entries are defined),
 the tableau $v^-$ obtained from $v$ by replacing the entry $r$
in position $(i,j)$ with $r-1$ is conjugate-semistandard. Moreover $u \preceq_\maj v^-
\preceq_\maj v$, so by choice of $u$, we have $u=v^-$ and so $v^- \not\in \mathcal{T}_i$.
By replacing $v$ by $v^-$ and
repeating this process we eventually obtain a closed conjugate-semistandard
tableau family; by Proposition~\ref{prop:type}, this family has a well-defined
type. 
 We repeat this argument for each non-closed conjugate-semistandard
tableau family within the tuple. Replacing each such $\mathcal{T}_i$ by the  closed conjugate-semistandard tableau family  constructed above yields a closed conjugate-semistandard
tableau family tuple. By construction, this has smaller type.
\end{proof}

It immediately follows that each of the families
in a conjugate-semistandard tableau family tuple of minimal partition
type is closed.

\section{Models for generalized Foulkes modules}
\label{sec:models}

\newcommand{\MOmega}{\mathcal{M}}
\newcommand{\NOmega}{\mathcal{N}}

\subsection{Preliminaries}
\label{subsec:models_preliminary}
Recall that $\mu$ and $\nu$ are
partitions of $m$ and $n$, respectively. Let~$\Omega$ be a set of size $mn$.
Let~$\MOmega$ be the set of $\mu$-tableaux with distinct  entries from
$\Omega$. 
Let $\NOmega(\Omega)$ be the set of $\nu$-tableaux $\bT$ with entries
from $\MOmega$ such that the union of the sets of entries of the $\mu$-tableaux in
$\bT$ is $\Omega$. 
(Generally we use capital letters to denote tableaux
whose entries are themselves tableaux and bold capital letters for elements of $\NOmega(\Omega)$.) 
For brevity we  write $\NOmega$ for $\NOmega(\Omega)$ in this section.

Given $\phi \in S_\Omega$ and $\bT \in \NOmega$ we define $\bT \phi$ by making
$\phi$ act on each entry of each $\mu$-tableau in $\bT$: thus $(i,j)(\bT \phi) = ((i,j)\bT)\phi$
for each $(i,j) \in [\nu]$, 
where, in turn, if $(i,j)T = t$ then $(a,b)(t \phi) = ((a,b)t)\phi$
for each $(a,b) \in [\mu]$.
Given $\sigma = (\sigma_{(i,j)} : (i,j) \in [\nu]) \in \RPP(\mu)^{\timesn}$, 
and a tableau $\bT\in \NOmega$, let
$\bT \star \sigma \in \NOmega$ be defined
by $(i,j) (\bT \star \sigma) = (i,j)\bT \cdot \sigma_{(i,j)}$.
We define
$\bT \star \pi$ for $\pi \in \CPP(\mu)^{\timesn}$ similarly. 
We define $\sgn(\pi) = \prod_{(i,j) \in [\nu]} \sgn(\pi_{(i,j)})$.
Observe that the action of $S_\Omega$ on $\NOmega$ commutes with that of $\RPP(\mu)^{\timesn}$ and $\CPP(\mu)^{\timesn}$.

Define
$N 
= \langle \{ \bT \} : \bT \in \NOmega \rangle_\mathbb{Q}$ 
and $\widetilde{N} 
= \langle |\bT| : \bT \in \NOmega \rangle_\mathbb{Q}$. We define $\mathbb{Q}S_\Omega$-submodules $R$ and $C$ of $N$ by
\begin{align*} 
R &= \bigl\langle \{\bT \star \sigma \} - \{\bT \} : \bT \in \NOmega,\, \sigma \in \RPP(\mu)^\timesn
\bigr\rangle_\mathbb{Q}, \\
C &= \bigl\langle \{ \bT \star \pi \} - \sgn(\pi) \{ \bT \} : \bT \in \NOmega,\, \pi \in \CPP(\mu)^\timesn
\bigr\rangle_\mathbb{Q} 
\end{align*}
and $\mathbb{Q}S_\Omega$-submodules $\widetilde{R}$ and $\widetilde{C}$ of $\widetilde{N}$ by
\begin{align*}
\widetilde{R} &= \bigl\langle |\bT \star \sigma | - |\bT | : \bT \in \NOmega,\, \sigma \in \RPP(\mu)^\timesn
\bigr\rangle_\mathbb{Q}, \\
\widetilde{C} &= \bigl\langle |\bT \star \pi| - \sgn(\pi) | \bT | : \bT \in \NOmega,\, \pi \in \CPP(\mu)^\timesn
\bigr\rangle_\mathbb{Q} .
\end{align*}

Observe that if $\bT$, $\bU \in \NOmega$  and, for each $(i,j) \in [\nu]$, the
$\mu$-tableau entries $(i,j) \mathbf{T}$, $(i,j) \bU$ of $\bT$ and $\bU$
agree up to the order of their rows, then $\{\mathbf{T}\}$ and $\{\bU\}$ are congruent modulo $R$. Informally put, working modulo $R$
we may regard the $\mu$-tableau entries of tableaux in $\NOmega$ as 
 $\mu$-tabloids.

\subsection{Models}\label{subsec:models} 
To make this idea more precise we shall give explicit bases for $N/R$ and $\widetilde{N}/\widetilde{R} $
and explicit isomorphisms $N/R \cong (M^\mu \oslash M^\nu) \indmn$ and
$\widetilde{N}/\widetilde{R} \cong (M^\mu \oslash \widetilde{M}^\nu) \indmn$.
For this we must suppose that $\Omega$ is totally ordered; row-standard
and column-standard for $\mu$-tableaux then have their 
expected meanings from Definition~\ref{defn:rowcolstd}. 
If $s$ and $t$ are $\mu$-tableaux with disjoint sets of entries from $\Omega$, we
set $s <_\i t$ if
the least entry of $s$ is strictly smaller than the least entry of~$t$.

\begin{definition}\label{defn:NR}
Let $\NR$ be the set of $\bT \in \NOmega$ such that each $\mu$-tableau
entry of $\bT$ is row-standard. Let $\NR_\row$ be the set of $\bT \in \NR$
such that $\bT$ is row-standard in the order $<_\i$. Let $\NR_\col$ be the
set of $\bT \in \NR$ such that $\bT$ is column-standard in the order $<_\i$.
\end{definition}

We  need a small extension of the construction
in \S\ref{subsec:oslash}. Given $\nu$ a partition of $n$, a $\Q S_m$-module $W$  and a $\Q S_n$-module $Z$, 
let $S_\nu \le S_n$ denote the Young subgroup labelled by $\nu$ and let $W \oslash_\nu Z$ denote the
$\Q(S_m^n \rtimes S_\nu)$-module obtained by restricting $W \oslash Z$
to $S_m^n \rtimes S_\nu$. We also need the characterization of induced
modules in Alperin \cite[Chapter~III, Corollary~3]{Alperin}: if $X$ is an $\Q G$-module
and $Y$ is a $\Q K$-module,
where $G \le K$ are finite groups, then \hbox{$Y \cong X\ind_G^K$} if and only if $Y\res_G$ has
a $\Q G$-submodule $W$ such that $W$ generates~$Y$ as a $\Q K$-module, $W$ is isomorphic
to $X$ as a $\Q G$-module, and $\dim Y = |K : G|\dim X $.
We denote the trivial and alternating $\Q S_{n}$-modules by $\Q_{S_n}$ and $\sgn_{S_{n}}$ respectively.

\begin{proposition}\label{prop:models}
Let $\Gnu = S_m^\timesn \rtimes S_\nu$. Identify $S_{mn}$ and $S_\Omega$ via the unique order preserving bijection
$\{1,\ldots,mn\} \rightarrow \Omega$. Under this identification there are isomorphisms
\begin{align*}
\mathrm{(i)}\ N/R &\cong \bigl(M^\mu \oslash_\nu \Q_{S_n}\bigr)\Ind_\Gnu^{S_{mn}} 
 \cong \bigl( M^\mu \oslash M^\nu \bigr) \Indmn, \\
\mathrm{(ii)}\ N/C &\cong \bigl(\widetilde{M}^\mu \oslash_\nu \Q_{S_n} )\Ind_\Gnu^{S_{mn}} 
 \cong \bigl( \widetilde{M}^\mu \oslash M^\nu \bigr) \Indmn, \\
\mathrm{(iii)}\ \widetilde{N}/\widetilde{R} &
 \cong \bigl( M^\mu \oslash_{\nu'} \sgn_{S_n}\bigr)\Ind_{G_{\nu'}}^{S_{mn}} 
 \cong \bigl( M^\mu \oslash \widetilde{M}^\nu \bigr) \indmn, \\
\mathrm{(iv)}\ \widetilde{N}/\widetilde{C} &\cong 
(\widetilde{M}^\mu \oslash_{\nu'} \sgn_{S_{n}} \bigr)\Ind_{G_{\nu'}}^{S_{mn}}   \cong
 \bigl( \widetilde{M}^\mu \oslash \widetilde{M}^\nu \bigr) \indmn.
\end{align*}
Moreover $N/R$ has $\bigl\{ \{ \bT \} + R : \bT \in \NR_\row \bigr\}$ as a basis
and $\widetilde{N}/\widetilde{R}$ has 
 $\bigl\{ |\bT| + \widetilde{R} : \bT \in \NR_\col \bigr\}$ as a basis.
\end{proposition}

\begin{proof}
Let $t$ be a $\mu$-tableau with distinct entries from $\{1,\ldots,m\}$. 
Let $u$ be a $\nu$-tableau with distinct entries from $\{1,\ldots,n\}$
whose rows are fixed setwise by $S_\nu \le S_n$. Then
$\langle \{u\} \rangle$ affords the trivial representation of the Young subgroup $S_\nu \le S_n$.
It is clear that $(\{t\} \otimes \cdots \otimes \{t\}) \otimes \{u\} \in M^\mu \oslash M^\nu$
generates $M^\mu \oslash_\nu \langle \{ u\} \rangle$ as a $\Q \Gnu$-module
and $M^\mu \oslash M^\nu$ as a $\Q (S_m \wr S_n)$-module.
Moreover, $\dim M^\mu \oslash_\nu \langle \{ u \} \rangle = (\dim M^\mu)^n$ and
$|S_m \wr S_n : \Gnu| = |S_n : S_\nu| = \dim M^\nu$, so
\[ \bigl( \dim M^\mu \oslash_\nu \langle \{ u \} \rangle \bigr) 
|S_m \wr S_n : \Gnu| = (\dim M^\mu)^n \dim M^\nu = \dim M^\mu \oslash M^\nu.\]
Hence, by the characterization of induced modules,  
\[ M^\mu \oslash M^\nu \cong
 (M^\mu \oslash_\nu \Q_{S_n}) \Ind_\Gnu^{S_m \wr S_n}\]
and so $(M^\mu \oslash M^\nu)\indmn \cong (M^\mu \oslash_\nu \Q_{S_n}) \ind_\Gnu^{S_{mn}}$.
This proves the second isomorphism in (i).

Let $\iota : \{1,\ldots,mn\} \rightarrow \Omega$ be the unique order preserving bijection.
For each $(i,j) \in [\nu]$, if $(i,j)$ is the $\ell$th box of $[\nu]$ in the lexicographic order 
on $[\nu]$, then set \hbox{$\Gamma_{(i,j)} = \{ (\ell-1)m+1,\ldots, \ell m\} \iota$}.
For each $(i,j) \in [\nu]$ let $t_{(i,j)}$ be the $\mu$-tableau obtained from
the fixed $\mu$-tableau~$t$ by
replacing each entry $r \in \{1,\ldots,m\}$ with the $r$th smallest entry of $\Gamma_{(i,j)}$.
(This choice is made for definiteness; any $\mu$-tableau with the same
set of entries may be used.)
For each $(a,b) \in [\mu]$, let 
\[ \Delta_{(a,b)} = \{ (a,b)t_{(i,j)} : (i,j) \in [\nu] \} \]
be the set of entries of the tableaux $t_{(i,j)}$ in position $(a,b)$.
Let $H_\Omega$ be the subgroup of $\prod_{(a,b)\in [\mu]} S_{\Delta_{(a,b)}}$ consisting
of all permutations 
that permute the $\Gamma_{(i,j)}$ as blocks for their action.
(These definitions are illustrated in the example following this proof.)
There are isomorphism of abstract groups
$H_\Omega \cong S_{\Delta_{(a,b)}} \cong S_n$
for each $(a,b) \in [\mu]$. 
Let $H_\nuOmega \le H_\Omega$ be the Young subgroup of $H_\Omega$ 
whose orbits on the blocks $\Gamma_{(i,j)}$ are
\[ \bigl\{ \Gamma_{(1,j)} : j \in \{1,\ldots, \nu_1\} \bigr\}, \ldots,
\bigl\{ \Gamma_{(k,j)} : j \in \{1,\ldots, \nu_k\} \bigr\} \]
where $k$ is the number of parts of $\nu$. The subgroup
\begin{equation}\label{eq:GnuOmega} 
G_\nuOmega = \bigl( \prod_{(i,j) \in [\nu]} S_{\Gamma_{(i,j)}} \bigr) \rtimes H_\nuOmega 
\end{equation}
of $S_\Omega$ is then conjugate to $\Gnu \le S_{mn}$, after identifying
$S_{mn}$ and $S_\Omega$ via the bijection $\iota$.
Let $\bVV \in \NOmega$  be the $\nu$-tableau with entries from $\MOmega$ defined by 
$(i,j)\bVV = t_{(i,j)}$. 
It is clear that $\langle \{\bVV\} + R \rangle$ affords
the trivial representation of the subgroup
\[ \bigl( \prod_{(i,j) \in [\nu]} \Stab_{S_{\Gamma_{(i,j)}}} \{ t_{(i,j)} \} \bigr) \rtimes H_\nuOmega \]
of $G_\nuOmega$. Therefore the $\Q G_\nuOmega$-submodule of $N/R$ generated
by $\{\bVV\} +R$ is isomorphic to \hbox{$M^\mu \oslash_\nu \Q_{S_n}$}, after identifying
$G_\nu$ and $G_\nuOmega$ via the bijection $\iota$.
The first isomorphism in (i) now follows from the characterization of induced modules:
it is given explicitly by 
\begin{equation}\label{eq:iso_i}  \{ \bVV \} + R  \mapsto \{ t \} \otimes \cdots \otimes \{t \} \otimes \{u\}.
\end{equation}
This completes the proof of~(i).

Now define
\[ \NR_X = \big\{ \bU \in \NR : \text{$(i,j)\bU$ has distinct
entries from $\Gamma_{(i,j)}$ for all $(i,j) \in [\nu]$}\bigr\} \] 
and let $X$ be the  $\Q G_\nuOmega$-submodule of $N/R$ generated by $\{\bVV\} + R$.
It follows from the choice of subsets $\Gamma_{(i,j)}$ 
that  $\NR_X \subseteq \NR_\row$ and that the
elements $\{ \bU \} + R $ for $\bU \in \NR_X$ form a basis for $X$ that
is permuted transitively by $G_\nuOmega$.
(This basis will be used later in the proof of Proposition~\ref{lemma:basis}.)
Considering the vector space decomposition $N/R = \bigoplus_\phi X \phi$
where $\phi$ runs over a set of coset representatives for the cosets
of $G_\nuOmega$ in $S_\Omega$, we see that $N/R$ has a basis as claimed.

The remaining parts can be proved similarly. For instance, for (iii), replace
$\langle \{u\} \rangle$ with $\langle |u| \rangle$ 
affording the sign representation of the Young subgroup $S_{\nu'} \le S_n$,
replace $\{\bVV\} + R$ with $|\bVV| + \widetilde{R}$, and replace
 $H_\Omega$ with a Young subgroup whose orbits on the blocks $\Gamma_{(i,j)}$ are
\[ \bigl\{ \Gamma_{(i,1)} : i \in \{1,\ldots, \nu'_1\} \bigr\}, \ldots,
\bigl\{ \Gamma_{(i,c)} : i \in \{1,\ldots, \nu'_c\} \bigr\}, \]
where $c = \nu_1$. The isomorphism is given by 
\begin{equation}\label{eq:iso_ii}
|\bVV| + R  \mapsto \{ t \} \otimes \cdots \otimes \{t \} \otimes |u|.
\end{equation}
\end{proof}

\newcommand{\Ten}{10}\newcommand{\Jack}{11}\newcommand{\Queen}{12}
To illustrate the proof of (i), we take $\mu = (2,1)$ and $\nu = (3,1)$.
Let $\Omega = \{1,\ldots, 12\}$ with the natural order; thus $\iota$ is the identity map.
If we take $t = \young(12,3)\rule{0pt}{16pt}$ then
\[ t_{(1,1)} = \young(12,3)\,,\quad t_{(1,2)} = \young(45,6)\,,\quad t_{(1,3)} = \young(78,9),
\quad
t_{(2,1)} = \young(\Ten\Jack,\Queen) \]
and $\bVV \in \NOmega$ is the $(3,1)$-tableau with these $(2,1)$-tableaux as entries.
The blocks $\Gamma_{(i,j)}$ for $(i,j) \in [\nu]$ are the entries of these $(2,1)$-tableaux,
and $\Delta_{(1,1)} = \{1,4,7,10\}$, $\Delta_{(1,2)} = \{2,5,8,11\}$, $\Delta_{(2,1)} =
\{3,6,9,12\}$. The subgroup
$H_\Omega \le S_{\{1,4,7,10\}} \times S_{\{2,5,8,11\}} \times S_{\{3,6,9,12\}}$ is generated by
the permutations
$(1,4,7,10)(2,5,8,11)(3,6,9,12)$ and $(1,4)(2,5)(3,6)$. It permutes
the $\Gamma_{(i,j)}$ as blocks for its action. The subgroup $H_{\nuOmega}$
is generated by the permutations 
$(1,4,7)(2,5,8)(3,6,9)$ and $(1,4)(2,5)(3,6)$ and
$G_\nuOmega 
= (S_{\{1,2,3\}} \times S_{\{4,5,6\}} \times S_{\{7,8,9\}} \times S_{\{10,11,12\}}) \rtimes H_{\nuOmega}$.

The set $\NR_X$
consists of all $\young(stu,v) \in \NOmega\rule{0pt}{15pt}$ such that
\begin{align*}
& s \in \Bigl\{\, \syoung{(12,3)}\,,\ \syoung{(13,2)}\,,\ \syoung{(23,1)}\, \Bigr\},\
 t \in \Bigl\{\, \syoung{(45,6)}\,,\ \syoung{(46,5)}\,,\ \syoung{(56,4)}\, \Bigr\}, \\
& u \in \Bigl\{\, \syoung{(78,9)}\,,\ \syoung{(79,8)}\,,\ \syoung{(89,7)}\, \Bigr\},\ 
v \in \Bigl\{\, \syoung{(\Ten\Jack,\Queen)}\,,\ \syoung{(\Ten\Queen,\Jack)}\,,
\ \syoung{(\Jack\Queen,\Ten)}\, \Bigr\}\, . \\[-13pt]
\end{align*}
Thus the $\Q G_\nuOmega$-submodule $X$ of $N/R$ generated by $\{ \bVV \} + R$ is isomorphic
to $M^{(2,1)} \oslash_\nu \Q_{S_4}$ and has $\bigl\{ \{ \bU \} + R : \bU \in \NR_X \bigr\}$
as a basis permuted transitively by $G_\nuOmega$.

\subsection{Maps}
We noted in \S\ref{subsec:oslash} that  $W \oslash Z$ is functorial
in $W$ and $Z$. From the canonical projections  $\widetilde{M}^\lambda \rightarrow S^\lambda$
and $M^\lambda \rightarrow \widetilde{S}^\lambda$ defined in \S\ref{subsec:Specht} we obtain
the canonical projections  made explicit in the following proposition. (The two projections
not given are the only two we do not use.)

\begin{proposition}\label{prop:modelHoms}
Let $\bUT \in \NOmega$. The homomorphisms \emph{(i)} $\widetilde{N}/\widetilde{R} \rightarrow N/R$,
\emph{(ii)} $N/R \rightarrow \widetilde{N}/\widetilde{R}$, 
\emph{(iii)} $\widetilde{N}/\widetilde{R} \rightarrow \widetilde{N}/\widetilde{C}$,
\emph{(iv)} $N/R \rightarrow N/C$, 
\emph{(v)} $\widetilde{N}/\widetilde{C} \rightarrow \widetilde{N}/\widetilde{R}$,
\emph{(vi)} $N/C \rightarrow N/R$,
corresponding to the compositions
\begin{align*}
\mathrm{(i)}\ (M^\mu  \oslash \widetilde{M}^\nu) \indmn &\twoheadrightarrow (M^\mu \oslash S^\nu)
\indmn \hookrightarrow (M^\mu \oslash M^\nu) \indmn \\
\mathrm{(ii)}\ (M^\mu \oslash M^\nu) \indmn &\twoheadrightarrow 
(M^\mu \oslash \widetilde{S}^\nu) \indmn
 \hookrightarrow (M^\mu \oslash \widetilde{M}^\nu) \indmn \\
\mathrm{(iii)}\ (M^\mu  \oslash \widetilde{M}^\nu) \indmn &\twoheadrightarrow (\widetilde{S}^\mu \oslash 
\widetilde{M}^\nu) 
\indmn \hookrightarrow (\widetilde{M}^\mu 
\oslash \widetilde{M}^\nu) \indmn \\
\mathrm{(iv)}\ (M^\mu  \oslash M^\nu) \indmn&\twoheadrightarrow 
(\widetilde{S}^\mu \oslash M^\nu) \indmn \hookrightarrow (\widetilde{M}^\mu 
\oslash M^\nu) \indmn \\
\mathrm{(v)}\ (\widetilde{M}^\mu  \oslash \widetilde{M}^\nu) \indmn&\twoheadrightarrow 
(S^\mu \oslash \widetilde{M}^\nu) \indmn \hookrightarrow (M^\mu 
\oslash \widetilde{M}^\nu) \indmn \\
\mathrm{(vi)}\ (\widetilde{M}^\mu  \oslash M^\nu) \indmn&\twoheadrightarrow 
(S^\mu \oslash M^\nu) \indmn \hookrightarrow (M^\mu \oslash M^\nu) \indmn 
\end{align*}
are defined 
on the generators $|\bUT|+\widetilde{R}$, $\{\bUT\} + R$, $|\bUT| + \widetilde{R}$, $\{\bUT\} + R$, $|\bUT| +  \widetilde{C}$, $\{ \bUT\} + C$
of their respective domains by
\leqnomode
\begin{center}
\begin{minipage}{3in}
\begin{align*}
\phantom{\qquad}|\bUT| + \widetilde{R} &\;\longmapsto \sum_{\tau \in \CPP(\nu)} \{\bUT \pp \tau\}
\sgn(\tau) + R, 
\\
\phantom{\qquad}\{\bUT\} + R &\;\longmapsto \sum_{\tau \in \RPP(\nu)} |\bUT \pp \tau| + \widetilde{R},
\\
\phantom{\qquad}|\bUT| + \widetilde{R} &\; \longmapsto \sum_{\sigma \in \RPP(\mu)^\timesn} |\bUT \star \sigma| + \widetilde{C},  \\ 
\phantom{\qquad}\{\bUT\} + R &\;\longmapsto \sum_{\sigma \in \RPP(\mu)^\timesn} \{\bUT \star \sigma\} + C, 
\\
\phantom{\qquad}|\bUT| + \widetilde{C} &\;\longmapsto \sum_{\pi \in \CPP(\mu)^\timesn} |\bUT \star \pi|\sgn(\pi) + \widetilde{R}, 
\\
\phantom{\qquad}\{\bUT\} + C &\;\longmapsto \sum_{\pi \in \CPP(\mu)^\timesn} \{ \bUT \star \pi \} \sgn(\pi) + R.
\end{align*}
\end{minipage}
\end{center}
\end{proposition}

\begin{proof}
We prove (iii) to illustrate the action of the group $\RPP(\mu)^\timesn$. 
There is no loss of generality in taking $\bUT$ to be the tableau $\bVV \in \NOmega$
 defined in the proof of Proposition~\ref{prop:models}.
Also, as in this proof, let $t$ be a $\mu$-tableau
with distinct entries from $\{1,\ldots,m\}$ and let $u$ be
a $\nu$-tableau with distinct entries from $\{1,\ldots,n\}$ such that
$\langle |u| \rangle$ affords the sign representation of $S_{\nu'}$.
 By  Proposition~\ref{prop:models}
there are isomorphisms $\widetilde{N}/\widetilde{R} 
\cong (M^\mu  \oslash \widetilde{M}^\nu) \indmn$
and $\widetilde{N}/\widetilde{C} \cong (\widetilde{M}^\mu  \oslash \widetilde{M}^\nu) \indmn$
defined in Equation~(\ref{eq:iso_ii}) by
$|\bVV| + \widetilde{R} \mapsto \{t\} \otimes \cdots \otimes \{t\} \otimes |u|$ 
and (similarly)
$|\bVV| + \widetilde{C} \mapsto |t| \otimes \cdots \otimes |t| \otimes |u|$, respectively.
 The image of $\{t\}$ under
the canonical surjection $M^\mu \rightarrow \widetilde{S}^\mu$ is
$\sum_{\sigma \in \RPP(\mu)} |t \pp \sigma| $.
Hence the induced map $\widetilde{N}/\widetilde{R} \rightarrow \widetilde{N}/\widetilde{C}$  satisfies
\[
|\bVV|+ \widetilde{R} \mapsto \sum_{\sigma \in \RPP(\mu)^\timesn} |\bVV \star \sigma | + \widetilde{C},
\]
since for $\sigma = (\sigma_{(i,j)} : (i,j) \in [\nu]) \in \RPP(\mu)^n$ we have $(i,j)(\bVV \star \sigma)=t_{(i,j)}\pp \sigma_{(i,j)}$.
\end{proof}

\subsection{Models for plethystic Specht modules}

In this subsection, we combine the results so far to give explicit models for the
$\Q S_{mn}$-modules $(S^\mu \oslash M^\nu) \indmn$ and $(S^\mu \oslash 
\widetilde{M}^\nu) \indmn$. These models are used in \S\ref{sec:toTuple}.

\begin{definition}\label{defn:Nrow}
Let $\NOmega_\row$ be the set of $\bT \in \NOmega$ such that each 
row of $\bT$ is strictly increasing under $<_\i$ and let $\NOmega_\col$
be the set of $\bT \in \NOmega$ such that each column of $\bT$ is strictly increasing
under $<_\i$.
\end{definition}

We note that $\NR_\row = \NR \cap \NOmega_\row$ and $\NR_\col = \NR \cap \NOmega_\row$.

\begin{definition}\label{defn:erowcol}
For each $\bT \in \mathcal{N}$, define 
\begin{align*} \erow(\bT) &=
\sum_{\pi \in \CPP(\mu)^n} \{ \bT \star \pi \} \sgn(\pi) + R \in N/R, \\ 
\ecol(\bT) &= \sum_{\pi \in \CPP(\mu)^n} |\bT \star \pi| \sgn(\pi) + 
\widetilde{R} \in \widetilde{N}/\widetilde{R}.\end{align*}
\end{definition}

The elements $\erow(\bT)$ and $\ecol(\bT)$ behave much like ordinary polytabloids. In particular, we have
\begin{equation}
\label{eq:SpechtCyclic2}
\erow(\bT) \phi = \erow(\bT \phi) \quad\text{and}\quad \ecol(\bT) \phi = \ecol(\bT \phi)
\end{equation}
for any $\phi \in S_{\Omega}$.
The following lemma shows that a linear relation between
polytabloids in $S^\mu$ implies linear relations
in the submodules of $N/R$ and $\widetilde{N}/\widetilde{R}$
isomorphic to $(S^\mu \oslash M^\mu) \indmn$ and $(S^\mu \oslash \widetilde{M}^\nu) \indmn$,
respectively.

\begin{lemma}\label{lemma:rel}
Let $\bT \in \mathcal{N}$. Let $(i,j) \in [\nu]$, let $t = (i,j)\bT$ and let
$\Gamma$ be the set of entries of $t$. Suppose that
$e(t) = \sum_v c_v e(v)$ where the sum is over $\mu$-tableaux $v$ with
distinct entries from $\Gamma$ and $c_v \in \Q$ for each $v$.
For each such $v$, let $\bT_v \in \mathcal{N}$ be obtained from
$\bT$ by replacing the entry~$t$ in position $(i,j)$ with $v$.

\begin{thmlist}
\item If $\bT \in \mathcal{N}_\row$ 
then $\erow(\bT) = \sum_v c_v \erow(\bT_v)$
and $\bT_v \in \mathcal{N}_\row$ for each $v$.

\item If $\bT \in \mathcal{N}_\col$
then $\ecol(\bT) = \sum_v c_v \ecol(\bT_v)$
and $\bT_v \in \mathcal{N}_\col$ for each $v$.
\end{thmlist}
\end{lemma}

\begin{proof}
We prove (i); the proof of (ii) is similar, using part (iii) of Proposition~\ref{prop:models}
rather than part (i).
For each $(i',j') \in [\nu]$, let
 $\Gamma_{(i',j')}$ be the set of entries of the $\mu$-tableau  $(i',j')\bT$.
Given any $\pi \in \CPP(\mu)^n$ there exists a unique permutation
\[ \tilde{\pi} \in \prod_{(i',j') \not= (i,j)} S_{\Gamma_{(i',j')}} \]
such that $(i',j')(\bT \star \pi) = \bigl( (i',j')\bT \bigr) \tilde{\pi}$
for all $(i',j') \not= (i,j)$. Note that $\sgn(\pi) = \sgn(\tilde{\pi}) \sgn(\pi_{(i,j)})$.
Moreover, since each $\bT_v$ agrees with $\bT$ in its positions
other than $(i,j)$, we  have
\[ (i',j')(\bT_v \star \pi) = \bigl( (i',j')\bT_v \bigr) \tilde{\pi} \] 
for all $(i',j') \not= (i,j)$.
Let $H = \{ \tilde{\pi} : \pi \in \CPP(\mu)^n \}$. By these remarks
\begin{align*}
e_\row(\bT) &= \sum_{\phi \in H}\sum_{\pi_{(i,j)} \in \CPP(\mu)} \{ \bT \star \pi_{(i,j)} \} 
\sgn(\pi_{(i,j)}) \phi \sgn(\phi)
+ R, \\
\intertext{and}
\sum_v c_v e_\row(\bT_v) &= \sum_v \sum_{\phi \in H} \sum_{\pi_{(i,j)} \in \CPP(\mu)} \{ \bT_v \star \pi_{(i,j)} \}
 \sgn(\pi_{(i,j)})\phi \sgn(\phi) + R,
\end{align*}
where, by a small abuse of notation, we 
regard $\pi_{(i,j)} \in \CPP(\mu)$ as an element of $\CPP(\mu)^n$ acting trivially on all $\mu$-tableaux not in position $(i,j)$ of a $\nu$-tableau. It therefore suffices to show that
\[ 
\sum_{\pi_{(i,j)} \in \CPP(\mu)} \!\!\!\!  \{ \bT \star \pi_{(i,j)} \} 
\sgn(\pi_{(i,j)})  + R \\
= \sum_{\pi_{(i,j)} \in \CPP(\mu)} \sum_v c_v  \{ \bT_v \star \pi_{(i,j)} \} \sgn(\pi_{(i,j)}) + R. 
\]
By Proposition~\ref{prop:models}(i), the  $\{\bU\} + R$ for $\bU \in \NR_\row$
are a basis for $N/R$. Since $\bT \in \mathcal{N}_\row$, the coefficient of $\{\bU\} + R$ in either side
is zero unless $\{ (i',j')\bU \} = \{(i',j') \bT \}$ for each $(i',j') \not=(i,j)$.
(Since $\bU \in \NR_\row$, this determines $\bU$ up to the entry in position $(i,j)$.)
Suppose that $(i,j)\bU = u$.
Then the coefficient of $\{\bU\} +R$ in the left-hand side is the coefficient
of $\{u\}$ in~$e(t)$, and the coefficient of $\{\bU\} +R$ in the right-hand side is the
coefficient of $\{u\}$ in $\sum_v c_ve(v)$. These agree by hypothesis.
\end{proof}

By the Standard Basis Theorem (see \cite[Theorem~8.5] {James}), if $t$
is a $\mu$-tableau with distinct entries from $\Gamma \subseteq \Omega$ then
$e(t) \in S^\mu$ is an integral linear combination of polytabloids $e(u)$
for standard $\mu$-tableaux $u$ with distinct entries from $\Gamma$.
This motivates the following definition.

\begin{definition}\label{defn:NS}
Let $\NS$ be the set of $\bT \in \mathcal{N}$ such that all $\mu$-tableau entries
of $\bT$ are standard. 
 Let $\NS_\row = \NS \cap \NOmega_\row$ and let $\NS_\col = \NS \cap \NOmega_\col$.
\end{definition}

\begin{lemma}{\ }\label{lemma:basis}
\begin{thmlist}
\item The set $\mathcal{B}_\row=\{ \erow(\bS) : \bS \in \NS_{\row} \}$ 
is a basis for a submodule of $N/R$ isomorphic to \linebreak $(S^\mu \oslash M^\nu)\indmn$.

\item The set $\mathcal{B}_\col=\{ \ecol(\bS)  : \bS \in \NS_{\col} \}$ 
is a basis for a submodule of $\widetilde{N}/\widetilde{R}$ isomorphic to \linebreak $(S^\mu \oslash
\widetilde{M}^\nu)\indmn$.
\end{thmlist}
\end{lemma}

\begin{proof}
We prove (i); 
the proof of~(ii) is similar.
Let $\bVV \in \mathcal{N}$ be as constructed in the proof of Proposition~\ref{prop:models}.
Let $V$ be 
the $\Q S_\Omega$-submodule of $N/R$ generated by $\erow(\bVV)$.
By Proposition~\ref{prop:modelHoms}(vi), $V$
is isomorphic to $(S^\mu \oslash M^\nu) \indmn$. 
It remains to show that $V$ has a basis $\mathcal{B}_\row$ as claimed.

Let 
$G_\nuOmega$ be as defined in~\eqref{eq:GnuOmega} in the proof of Proposition~\ref{prop:models}.
(Recall that, after identifying $S_{mn}$ with $S_\Omega$, $G_\nuOmega$
is conjugate to $S_m^{\timesn} \rtimes S_\nu$.)
Let $\NR_X$ be as defined in the proof of Proposition~\ref{prop:models}.
Recall that
$\bigl\{ \{ \bU \} + R : \bU \in \NR_X \bigr\}$ is a basis for a $\Q G_\nuOmega$-submodule
$X$ of $N/R$ isomorphic to $M^\mu \oslash_\nu \Q_{S_n}$ (after identifying
$G_\nu$ with $G_{\nuOmega}$), and that the basis
elements are permuted transitively by $G_{\nuOmega}$.
By definition of the sets $\Gamma_{(i,j)}$ in the construction of $\bVV$,
we have $\NR_X \subseteq \NOmega_\row$.

Let $Z$ be the $\Q G_\nuOmega$-submodule of $X$ generated by $\erow(\bVV)$.
By Propositions~\ref{prop:models}(i) 
and~\ref{prop:modelHoms}(vi),~$Z$ is isomorphic to $S^\mu \oslash_\nu \Q_{S_n}$
(after identifying $G_\nu$ with $G_\nuOmega$).
By~\eqref{eq:SpechtCyclic2} we have
$\erow(\bVV) \psi = \erow(\bVV \psi)$
for $\psi \in G_\nuOmega$. Since the top group $H_\nuOmega$ in the definition of $G_\nuOmega$
acts trivially on $X$, we may assume
that $\psi \in \prod_{(i,j) \in [\nu]} S_{\Gamma_{(i,j)}}$
and so each $\mu$-tableau entry $(i,j) (\bVV\psi)$ has the same set of entries as $(i,j)\bVV$. Thus
$(i,j)(\bVV\psi) $ has distinct entries from $\Gamma_{(i,j)}$ for all $(i,j) \in [\nu]$.
By repeated applications of
Lemma~\ref{lemma:rel} and the Standard Basis Theorem we may express each $\erow(\bVV \psi)$ as a linear
combination of elements $\erow(\bS)$ for $\bS \in \mathcal{NS}_\row \cap \NR_X$.
Hence the elements $e_\row(\bS)$ for $\bS \in \NS_\row \cap \NR_X$ span $Z$.
By dimension counting, using the isomorphism $Z \cong S^\mu \oslash_\nu \Q_{S_n}$, we see
that they are linearly independent.

Finally since $N/R$ is induced from $X$, there is a vector
space decomposition
\[ \label{eq:dsum} N/R = \bigoplus_\phi X \phi \]
where $\phi$ runs
over a set of coset representatives for $G_\nuOmega$ in $S_\Omega$.
The result now follows because $\phi$ sends the basis 
$\{ e_\row(\bS) : \bS \in \NS_\row \cap \NR_X\}$ of $Z$ to a basis of the relevant subspace of $N/R$.
\end{proof}


\section{Homomorphisms defined using conjugate-semistandard tableau family tuples}
\label{sec:homs}

In this section we use the models defined in \S\ref{sec:models} 
to define the homomorphisms needed to prove Theorem~\ref{thm:main}.

\subsection{Tableaux from conjugate-semistandard tableau family tuples}\label{subsec:tableaux}
Let $\mathcal{T} = (\mathcal{T}_1, \ldots, \mathcal{T}_c)$ be a conjugate-semistandard tableau family
tuple of shape $\mu^\kappa$ and type~$\lambda$, as in Theorem~\ref{thm:main}.
For each $i \in \{1,\ldots,c\}$ and each $j \in \{1,\ldots, \kappa_i\}$,
let $s_{(i,j)}$ be the $j$th smallest conjugate-semistandard tableau
in $\mathcal{T}_i$ under the total order $<$ defined in Definition~\ref{defn:colorder}.
Let~$S$ be the $\kappa$-tableau defined by $(i,j)S = s_{(i,j)}$ for each $(i,j)$.
Each entry of~$S$ is a $\mu$-tableau with entries from $\N$; for each $i \in \N$ the total number of 
entries equal to $i$ that appear is $\lambda_i'$. 

Let $\Omega^\lambda$ be the set of symbols defined in \S\ref{subsec:symbols}.
Let $\Sind$ be a $\kappa$-tableau obtained by appending indices to the entries of each $\mu$-tableau $(i,j)S$ so that the set of all the entries of the $\mu$-tableaux in $\Sind$ is $\Omega^\lambda$
and so that each $\mu$-tableau in~$\Sind$ has strictly increasing columns in the 
total order on $\Omega^\lambda$ defined in \S\ref{subsec:symbols}.
For definiteness we fix
the following procedure: start with the $\mu$-tableau $(1,1)S$, and continue in lexicographic
order of the boxes of $\nu$,
finishing with the $\mu$-tableau $(c,\kappa_c)S$, appending indices within 
each $\mu$-tableau in lexicographic order of the boxes of $\mu$. 
(For an example see Example~\ref{ex:ex}.)

Let $\bTT = \Sind'$ if $m$ is even and let $\bTT = \Sind$ if $m$ is odd. Recall that $\nu = \kappa'$
if $m$ is even, and $\nu = \kappa$ if $m$ is odd.
Thus $\bTT$ is a $\nu$-tableau in the set $\mathcal{N}(\Omega^\lambda)$, as defined in \S\ref{subsec:models_preliminary}.

\subsection{Homomorphisms}\label{subsec:homs}
Identify
$S_{mn}$ with $S_{\Omega^\lambda}$
using the unique order preserving bijection $\{1,\ldots,mn\} \rightarrow \Omega^\lambda$
and identify the $\Q S_{\Omega^\lambda}$-modules $N/R$, $\widetilde{N}/\widetilde{R}$,
$N/C$, $\widetilde{N}/\widetilde{C}$ defined in \S\ref{sec:models} with the $\Q S_{mn}$-modules 
$(M^\mu \oslash M^\nu) \indmn$, $(M^\mu \oslash \widetilde{M}^\nu) \indmn$,
$(\widetilde{M}^\mu \oslash M^\nu) \indmn$, $(\widetilde{M}^\mu \oslash \widetilde{M}^\nu) \indmn$
by four fixed isomorphisms. (These isomorphisms exist by Proposition~\ref{prop:models}.)

Let $t_\lambda$ be the $\lambda$-tableau with distinct entries from $\Omega^\lambda$ defined
in \S\ref{subsec:symbols}. Note that $|t_\lambda|$ generates $\widetilde{M}^\lambda$ as a $\Q S_{\Omega^\lambda}$-module.
Let $b_\lambda = \sum_{\theta \in I(\lambda)} \sgn(\theta) \theta$, where $I(\lambda)$ is the index
permuting group defined in \S\ref{subsec:symbols}.
We define homomorphisms $\f$ and $\g$ by
\begin{align*}
&
\f : \widetilde{M}^\lambda \rightarrow 
(M^\mu \oslash \widetilde{M}^\nu) \Indmn,
\quad |t_\lambda| f_{(\mathcal{T}_1,\ldots,\mathcal{T}_c)} = |\bTT| b_\lambda + \widetilde{R},\\
&
\g : \widetilde{M}^\lambda \rightarrow 
(M^\mu \oslash M^\nu) \Indmn, 
\quad |t_\lambda| g_{(\mathcal{T}_1,\ldots,\mathcal{T}_c)} = \{\bTT\} b_\lambda + R.
\end{align*}
That these are well-defined homomorphisms follows from the characterisation of induced modules.

Define
\begin{align*} 
&
\overf : \widetilde{M}^\lambda \rightarrow 
(\widetilde{S}^\mu \oslash S^\nu) \Indmn, \\ 
& 
\overg : \widetilde{M}^\lambda \rightarrow 
(\widetilde{S}^\mu \oslash \widetilde{S}^\nu) \Indmn
\end{align*}
by composing $\f$ 
with the canonical surjection
$(M^\mu \oslash \widetilde{M}^\nu) \indmn \twoheadrightarrow (M^\mu \oslash S^\nu) 
\indmn \twoheadrightarrow
(\widetilde{S}^\mu \oslash S^\nu) \indmn$ and composing
$\g$
with the canonical surjection
$(M^\mu \oslash M^\nu) \indmn \twoheadrightarrow (M^\mu \oslash \widetilde{S}^\nu) \indmn \twoheadrightarrow
(\widetilde{S}^\mu \oslash \widetilde{S}^\nu) \indmn$, respectively.
We remind the reader that
these definitions require the four fixed isomorphisms.
For example, the image of the homomorphism $\overf$ is, strictly speaking,
a submodule of $N/C$; to obtain $\overf$ as defined above we must identify $N/C$ with
$(\widetilde{M}^\mu \oslash M^\nu) \indmn$. 

We have
the following explicit description of the homomorphisms $\overf$ 
and~$\overg$. 

\begin{proposition}\label{prop:homsExplicit}
The images of $|t_\lambda| \in \widetilde{M}^\lambda$ under the homomorphisms
$\overf$ 
and $\overg$ are 
\begin{align*}
|t_\lambda|\overf 
&= \sum_{{\theta \in I(\lambda) \atop \tau \in \CPP(\nu)} \atop \sigma \in \RPP(\mu)^{\timesn}}
 \bigl\{ (\bTT\theta \pp \tau) \star \sigma\bigr) \bigr\} \sgn(\theta)\sgn(\tau) + C, \\
|t_\lambda|\overg 
&= \sum_{{\theta \in I(\lambda) \atop \tau \in \RPP(\nu)} \atop \sigma \in \RPP(\mu)^{\timesn}}
 \bigl| (\bTT\theta \pp \tau) \star \sigma \bigr| \sgn(\theta) + \widetilde{C},
\end{align*}
respectively.
\end{proposition}

\begin{proof}
This is immediate from Proposition~\ref{prop:modelHoms}.
\end{proof}

We end this section with an example of these homomorphisms.

\begin{example}\label{ex:ex}
Take $m=3$, $n=5$, $\mu = (2,1)$ and $\nu = (4,1)$. Take
the conjugate-semistandard tableau family tuple
\[ \mathcal{T} = \Bigl(  \Bigl\{  \hsss\young(12,1)\sp\young(12,2)\sp\young(13,1)\sp\young(13,2)\hsss \Bigr\}, 
\Big\{  \hsss \young(12,1)\hsss \Bigr\} \Bigr) \]
of shape $(2,1)^{(4,1)}$ and type $(3^2,2^3,1^3)$. (We have written tableaux in the total order $<$
defined in Definition~\ref{defn:colorder}.)
Appending indices to the tableaux in $\mathcal{T}$ we obtain
\[  \setlength{\arrayrulewidth}{.06em}
\makeatletter
\def\y@vr{\vrule height0.8\y@b@xdim width\y@linethick depth 0.3\y@b@xdim}
\def\y@setdim{%
  \ify@autoscale%
   \ifvoid1\else\typeout{Package youngtab: box1 not free! Expect an
     error!}\fi%
   \setbox1=\hbox{A}\y@b@xdim=1.8\ht1 \setbox1=\hbox{}\box1%
  \else\y@b@xdim=\y@boxdim \advance\y@b@xdim by -2\y@linethick
  \fi}
\makeatother
\bTT = \begin{tabular}{|c|c|c|c|}  \hline
$\young(\oa\ta,\ob)$ & $\young(\oc\tb,\tc)$ & $\young(\od\da,\oe)$ & $\young(\of\db,\td)$\rule[-17pt]{0pt}{40pt} \\
\cline{1-4}
$\young(\og\te,\oh)$\rule[-17pt]{0pt}{40pt} \\
\cline{1-1}
\end{tabular} \in \mathcal{N}(\Omega^{(3^2,2^3,1^3)}).
\]
The corresponding homomorphism 
\[ \overg : \widetilde{M}^{(3^2,2^3,1^3)} \rightarrow 
(\widetilde{S}^{(2,1)} \oslash \widetilde{S}^{(4,1)}) \Ind_{S_3 \wr S_5}^{S_{15}} \subseteq 
(\widetilde{M}^{(2,1)} \oslash
\widetilde{M}^{(4,1)}) \Ind_{S_3 \wr S_5}^{S_{15}} \cong \widetilde{N}/\widetilde{C}\]
sends the generator $|t_{(3^2,2^3,1^3)}|$ to
\[
\makeatletter
\def\y@vr{\vrule height0.8\y@b@xdim width\y@linethick depth 0.3\y@b@xdim}
\def\y@setdim{%
  \ify@autoscale%
   \ifvoid1\else\typeout{Package youngtab: box1 not free! Expect an
     error!}\fi%
   \setbox1=\hbox{A}\y@b@xdim=1.8\ht1 \setbox1=\hbox{}\box1%
  \else\y@b@xdim=\y@boxdim \advance\y@b@xdim by -2\y@linethick
  \fi}
\makeatother
 \setlength{\arrayrulewidth}{.08em}
\mathlarger{\sum}_{{{\theta \in I(3^2,2^3,1^3)} \atop \tau \in \RPP(4,1)} \atop \scriptscriptstyle
\sigma \in \RPP(2,1)^{5}}
\sgn(\theta)  \hsss \left|\,\left(\, \hsss
\begin{tabular}{|c|c|c|c|}  \hline
$\young(\oa\ta,\ob)$ & $\young(\oc\tb,\tc)$ & $\young(\od\da,\oe)$ & $\young(\of\db,\td)$\rule[-17pt]{0pt}{40pt} \\
\cline{1-4}
$\young(\og\te,\oh)$\rule[-17pt]{0pt}{40pt} \\
\cline{1-1}
\end{tabular} 
\hsss\, \theta \pp \tau \! \right) \star \sigma \right|{}+ \widetilde{C}.
 \]
The proof of Proposition~\ref{prop:nonzero} shows that the coefficient of $|\bTT| + \widetilde{C}$
in $|t_\lambda| \overg$ is $2$; correspondingly, 
$\theta = (\oa,\og)(\ob,\oh)(\ta,\te)$ swaps the two $(2,1)$-tableaux
in the first column of each columnar 
tabloid on the right-hand side, so
$ \sgn(\theta)  |(\bTT \theta \cdot \tau) \star \sigma| 
= |(\bTT \cdot \tau) \star \sigma|$ for all 
$\tau \in \RPP(4,1)$ and $\sigma \in \RPP(2,1)^5$.

We saw in Example~\ref{ex:ex} that $\mathcal{T}$ has minimal type. It
therefore follows from~Theorem~\ref{thm:main} and
Lemma~\ref{lemma:Mtilde} that the kernel of $\overg$ contains the Garnir relations
generating the kernel of
the canonical homomorphism $\widetilde{M}^{(3^2,2^3,1^3)} \rightarrow S^{(3^2,2^3,1^3)}$.
(In this case it is also possible to prove this in a more explicit way by
adapting the argument used in \cite[Proposition 5.2]{PagetWildonTwisted}.)
Hence we obtain a well-defined homomorphism from $S^{(3^2,2^3,1^3)}$ to the
module $(\widetilde{S}^{(2,1)} \oslash \widetilde{S}^{(4,1)}) \ind_{S_3 \wr S_5}^{S_{15}}$ 
isomorphic to $H_{(2,1)}^{(4,1)}$. We continue in Example~\ref{ex:excont} to
show that the multiplicity of $S^{(3^2,2^3,1^3)}$ in $H_{(2,1)}^{(4,1)}$ is $2$.
\end{example}

\section{The homomorphisms $\overf$ and $\overg$ are non-zero}
\label{sec:nonzero}

In this section we prove the following proposition on the homomorphisms defined in \S\ref{subsec:homs}

\begin{proposition}\label{prop:nonzero}
Let $(\mathcal{T}_1,\ldots,\mathcal{T}_c)$ be a conjugate-semistandard tableau family
tuple of shape $\mu^\kappa$ and type $\lambda$.
\begin{thmlist}
\item If $m$ is even then the homomorphism
\[ \overline{f}_{(\mathcal{T}_1,\ldots,\mathcal{T}_c)} : \widetilde{M}^\lambda \rightarrow
\bigl( \widetilde{S}^\mu \oslash S^\nu \bigr)\Indmn \]
is non-zero.
\item If $m$ is odd then the homomorphism
\[ \overline{g}_{(\mathcal{T}_1,\ldots,\mathcal{T}_c)} : \widetilde{M}^\lambda \rightarrow
\bigl( \widetilde{S}^\mu  \oslash \widetilde{S}^\nu \bigr)\Indmn \]
is non-zero.
\end{thmlist}
\end{proposition}


The proof of Proposition~\ref{prop:nonzero} 
uses Lemma~\ref{lemma:preorder} and the three further lemmas below.
Given a $\mu$-tableau $t$ with entries from $\Omega^\lambda$ let
$\D(t)$ be the \emph{deindexed} tableau with entries from~$\N$ obtained by
removing indices from the entries in $t$. Let $\mathcal{N}=\mathcal{N}(\Omega^\lambda)$ be 
as defined in \S\ref{subsec:models_preliminary}.

\begin{lemma}\label{lemma:rhopi}
Let $\bT \in \NOmega$. 
Let $\pi$, $\pi' \in \CPP(\mu)^{\timesn}$ and let
$\rho$, $\rho' \in \RPP(\nu)$. 
If $(\bT \pp \rho) \star \pi = (\bT \pp \rho') \star \pi'$ then $\rho = \rho'$ and
$\pi = \pi'$.
\end{lemma}

\begin{proof}
For each $(i,j) \in [\nu]$, 
the $\mu$-tableau in position $(i,j)\rho$ of $(\bT \pp \rho) \star \pi$
has the same set of entries as the $\mu$-tableau in position $(i,j)\rho'$
of $(\bT \pp \rho') \star \pi'$.
Since the sets of entries of the $\mu$-tableaux in any element of $\NOmega$ are disjoint,
it follows that $(i,j)\rho = (i,j)\rho'$ for all $(i,j) \in [\nu]$. Hence
$\rho = \rho'$. It is now clear that $\pi = \pi'$.
\end{proof}

\begin{lemma}\label{lemma:RPP}
Let $t_1, \ldots, t_n$ be $\mu$-tableaux with entries in $\Omega^\lambda$ such that each $||\D(t_i)||$ is conjugate-semistandard.
Let $\sigma_i \in \RPP(\mu)$ for each $i$.
Suppose that the multisets
$\bigl\{ ||\D(t_1 \pp \sigma_1)||, \ldots, ||\D(t_n \pp \sigma_n)|| \bigr\}$
and $\bigl\{||\D(t_1)||, \ldots, ||\D(t_n)||\bigr\}$ are equal. 
Then $\sigma_1 = \ldots = \sigma_n = \id$.
\end{lemma}

\begin{proof}
Let $\preceq$ be the usual total order on $\N$ and let $\preceq_\col$ be
the order induced by $\preceq$ 
on $\mu$-columnar tabloids with entries in $\N$, defined in Definition~\ref{defn:preorder}. 
By Lemma~\ref{lemma:preorder}(i), since each $D(t_i)$ is row-standard under $\preceq$,
we have $||D(t_i \pp \sigma_i)|| \preceq_\col ||D(t_i)||$ for each $i$.
By assumption, there exists $\phi \in S_n$ such that $||D(t_{i \phi})|| \preceq_\col ||D(t_{i} \pp \sigma_{i})||$
for all~$i$. By Lemma~\ref{lemma:alphabeta}, applied with the pre-order $\preceq_\col$,
we have $||D(t_i)|| \preceq_\col ||D(t_i \pp \sigma_i)||$ for all $i$. Hence, by
another application of Lemma~\ref{lemma:preorder}(i), $\sigma_i = \id$ for all $i$.
\end{proof}


\begin{lemma}\label{lemma:sgns}
Let $\psi \in I(\lambda)$, the group of index permutations, and
let $\bT \in \NOmega(\Omega^\lambda)$. 
If $\bT \psi \pp \rho = \bT$ where either $\rho \in \RPP(\nu)$ or $\rho \in \CPP(\nu)$ 
then 
\[ \sgn \psi = \begin{cases} 1 & \text{if $m$ is even} \\ \sgn \rho & \text{if $m$ is odd.} \end{cases}
\]
\end{lemma}

\begin{proof}
We suppose that $\rho \in \RPP(\nu)$: the other case is similar.
Let $\tilde{\rho}$ be the permutation of $\Omega^\lambda$ induced by $\rho$. 
By hypothesis $\tilde{\rho} = \psi^{-1} \in I(\lambda)$. If $x_y \tilde{\rho} = x_{z}$ with 
$y\not=z$ then
$x_y$ and $x_{z}$ are entries of distinct $\mu$-tableaux in the same row of $\bT$, say $s$ and $t$,
and $x'_{y'} \tilde{\rho}$ is an entry of $t$ whenever $x'_{y'}$ is an entry of $s$. Hence 
$\tilde{\rho}$ has exactly $m$ cycles of length $\ell$ for each $\ell$-cycle in $\rho$.
The lemma follows.
\end{proof}

We are now ready to prove Proposition~\ref{prop:nonzero}. We prove part (i) in full,
and then indicate the changes needed for (ii).

\begin{proof}[of Proposition~\ref{prop:nonzero}\emph{(i)}]
\newcommand{\pr}{\prime}
By Proposition~\ref{prop:homsExplicit}, we obtain a contribution to the coefficient of $\{\bTT\}+C$ 
in $|t_\lambda| \overline{f}_{(\mathcal{T}_1,\ldots,\mathcal{T}_c)}$ 
for each $\theta \in I(\lambda)$, $\tau \in \CPP(\nu)$ and $\sigma \in \RPP(\mu)^{\timesn}$
satisfying the condition $\{ (\bTT\theta \pp \tau) \star \sigma \} + C = \{ \bTT \} +C$. This condition
holds if and only if there exist $\pi \in \CPP(\mu)^{\timesn}$ and $\rho \in \RPP(\nu)$ such that
$\bigl( \bigl( (\bTT\theta \pp \tau) \star \sigma \bigr) \star \pi \bigr) \pp \rho = \bTT$.
If such $\pi$ and $\rho$ exist they are unique, by Lemma~\ref{lemma:rhopi}.
The coefficient is therefore
\[
\sum_{{{{{\theta \in I(\lambda) \atop \tau \in \CPP(\nu)} \atop \sigma \in \RPP(\mu)^{\timesn}} \atop 
\pi \in \CPP(\mu)^{\timesn}} \atop \rho \in \RPP(\nu)} \atop
((\bTTs \theta \pp \tau) \star \sigma) \star \pi) \cdot \rho = \bTTs}
 \hspace*{-12pt}\sgn(\theta)\sgn(\tau)\sgn(\pi) 
 \]
where  $\sgn(\pi) = \prod_{(i,j) \in [\nu]} \sgn(\pi_{(i,j)})$
is as defined in \S\ref{subsec:models_preliminary}.

Let $\theta,\tau,\sigma,\pi, \rho$ satisfy $((\bTT\theta \pp \tau) \star \sigma) \star \pi) \cdot \rho = \bTT$.
Let $t_{(i,j)} = (i,j)\bTT$. The tableau in position $(i,j)\tau$ of $\bTT\theta \pp \tau$
is $t_{(i,j)} \theta$. Since 
$\theta$ permutes indices, $\rho$ acts 
a position permutation on $\nu$-tableaux in $\mathcal{N}$ while leaving the multiset of 
$\mu$-tableaux entries invariant, and
$\pi$ acts as a column permutation on each $\mu$-tableau entry of a tableau in $\mathcal{N}$,
there is an equality of multisets
\begin{equation}
\bigl\{ ||\D(t_{(i,j)} \cdot \sigma_{(i,j)\tau})|| : (i,j) \in [\nu] \bigr\}
=  \bigl\{ ||\D(t_{(i,j)})|| : (i,j) \in [\nu] \bigr\}.
\label{eq:star}
\end{equation} 
Hence, by Lemma~\ref{lemma:RPP}, $\sigma_{(i,j)\tau} = \id$ 
for all $(i,j) \in [\nu]$, and so $\sigma = \id$.

Let $\preceq$ be the pre-order on $\mu$-tableaux with entries from $\Omega^\lambda$ defined by 
$s \preceq t$ if and only if $D(s) \le D(t)$, where $<$ is the total order in Definition~\ref{defn:colorder}.
The tableau $\bTT$ is column-standard under 
 $\preceq$.
Let $\preceq_\row$ be the pre-order on the set of tabloids $\{ \bT \}$ for $\bT \in \mathcal{N}$
defined by Definition~\ref{defn:preorder}(ii) using $\preceq$.
If $\tau\not= \id$ then, by Lemma~\ref{lemma:preorder}(ii), 
$\{ \bTT\theta \cdot \tau\} \prec_{\row} \{ \bTT \}$. Since the total order $<$
is preserved by column position permutations on $\mu$-tableaux
and $\preceq_{\row}$ depends only on the sets of entries in the rows of a tableau in $\mathcal{N}$,
it follows that
$\bigl\{  \bigl( (\bTT\theta \cdot \tau) \star \pi\bigr) \cdot \rho \bigr\} \prec_{\row} \{ \bTT \}$,
a contradiction. Therefore $\tau = \id$.

We now have $(\bTT \theta \star \pi) \cdot \rho = \bTT$. For each $(i,j) \in [\nu]$
there exists $(i,j') \in [\nu]$ such that $\bigl( (i,j)\bTT \theta \bigr) \cdot \pi_{(i,j)} =
(i,j')\bTT$.   Because the $\mu$-tableau entries of $\bTT$ are conjugate-semistandard after deindexing,
this shows that $\pi_{(i,j)}$  acts on the entry in position $(i,j)$ of $\bTT \theta$ by permuting indices.
Let $\tilde{\pi}_{(i,j)} \in I(\lambda)$
be the permutation induced by $\pi_{(i,j)}$ and let 
\[ \psi = \prod_{(i,j) \in [\nu]}
\theta\tilde{\pi}_{(i,j)}\in I(\lambda). \]
By Lemma~\ref{lemma:sgns} we have $\sgn \psi = 1$. It follows that $\sgn(\theta) = \sgn(\pi)$ and 
$\sgn(\theta)\sgn(\tau)\sgn(\pi) = \sgn(\theta)\sgn(\pi) = 1$. Hence the coefficient
of $\{ \bTT \}+C$ in $|t_\lambda|\overf$ is strictly positive.
\end{proof}

\begin{proof}[of Proposition~\ref{prop:nonzero}\emph{(ii)}]
A similar argument shows that the coefficient of $|\bTT|+\widetilde{C}$ in $t_\lambda
\overline{g}_{(\mathcal{T}_1,\ldots,\mathcal{T}_c)}$ is
\[
\sum_{{{{{\theta \in I(\lambda) \atop \tau \in \RPP(\nu)} \atop \sigma \in \RPP(\mu)^{\timesn}} \atop 
\pi \in \CPP(\mu)^{\timesn}} \atop \rho \in \CPP(\nu)} \atop
((\bTTs \theta \pp \tau) \star \sigma) \star \pi) \cdot \rho = \bTTs}
 \hspace*{-12pt}\sgn(\theta)\sgn(\pi) \sgn(\rho).
 \]
 Let $\theta,\tau,\sigma,\pi, \rho$ satisfy $((\bTT \theta \pp \tau) \star \sigma) \star \pi) \cdot \rho = \bTT$.
The argument for case~(i)
shows that~\eqref{eq:star} holds, and hence $\sigma_{(i,j)} = \id$ for all $(i,j) \in [\nu]$.
Let $\preceq$ be as defined in case~(i).
The tableau $\bTT$ is row-standard under the pre-order
 $\preceq$.  Hence by Lemma~\ref{lemma:preorder}(i),
if $\tau\not= \id$ then $|| \bTT\theta \cdot \tau|| \prec_{\col} || \bTT ||$.
A similar argument
now shows that
$\bigl|\bigl|  \bigl( (\bTT\theta \cdot \tau) \star \pi\bigr) \cdot \rho \bigr|\bigr| \prec_{\col} ||\bTT ||$,
a contradiction. Therefore $\tau = \id$. Now define $\psi$ as before, and apply Lemma~\ref{lemma:sgns}
to get $\sgn(\theta) \sgn(\pi) = \sgn(\psi) = \sgn(\rho)$, 
 showing that the coefficient
of $|\bTT|+\widetilde{C}$ in $|t_\lambda|\overf$ is strictly positive.
\end{proof}

Under certain technical hypotheses the
 homomorphisms $\overf$ and $\overg$
 for distinct conjugate-semistandard tableau family tuples
$\mathcal{T}$ are linearly independent; this is discussed in Section~\ref{sec:indep}.

\section{Conjugate-semistandard tableau families tuples from homomorphisms}
\label{sec:toTuple}

The aim of this section is to prove the following theorem.

\begin{theorem}\label{thm:toFamily}
Let $q$ be the number of conjugate-semistandard tableau family tuples of shape $\mu^\kappa$
and type $\lambda$. Then 
\[ \dim \Hom_{\Q S_{mn}}\bigl( S^\lambda, (S^\mu \oslash S^\nu)\Indmn 
\bigr) \le q.\]
\end{theorem}

The proof uses the results in \S\ref{sec:models}, taking
the set $\Omega$ to be $\Omega^\lambda$ with its usual order, namely
$y_{x} < y'_{x'}$ if and only if $y < y'$ or both $y=y'$ and $x < x'$.
Recall that, by Definition~\ref{defn:NS}, $\NS$ is the set of all $\nu$-tableaux
with standard $\mu$-tableaux entries such that the union of all the entries
of the $\mu$-tableaux is $\Omega^\lambda$, $\NS_\row$ is the subset
of $\NS$ of $\nu$-tableaux 
with strictly increasing rows under the order $<_\i$ (defined in Section~\ref{subsec:models}),
and $\NS_\col$ is the subset of $\NS$ of $\nu$-tableaux 
with strictly increasing columns under $<_\i$.
Let $\mathcal{B}_\row$ and $\mathcal{B}_\col$ be the bases of 
the submodules of $N/R$ and $\widetilde{N}/\widetilde{R}$
isomorphic to $(S^\mu \oslash M^\nu)\indmn$ and $(S^\mu \oslash \widetilde{M}^\nu)\indmn$, respectively, 
given in Lemma~\ref{lemma:basis}.
Recall  $\mathcal{B}_\row$ consists of all
\[ \erow(\bS) = \sum_{\pi \in \CPP(\mu)^{\timesn} } \{ \bS \star \pi \} \sgn(\pi) + R \]
where $\bS \in \NS_\row$ 
and $\mathcal{B}_\col$ consists of all
\[ \ecol(\bS) = \sum_{\pi \in \CPP(\mu)^{\timesn} } | \bS \star \pi | \sgn(\pi) + \widetilde{R} \]
where $\bS \in \NS_\col$. 

The key idea in the proof of Theorem~\ref{thm:toFamily}, when $m$ is even, is to 
pull back a $\Q S_{mn}$-homomorphism $S^\lambda \rightarrow
(S^\mu \oslash S^\nu) \indmn$ to a $\Q S_{\Omega^\lambda}$-homomorphism
$h : \widetilde{M}^\lambda \rightarrow \widetilde{N}/\widetilde{R}$
with image contained in the submodule of $\widetilde{N}/\widetilde{R}$ with
basis~$\mathcal{B}_\col$ isomorphic to $(S^\mu \oslash \widetilde{M}^\nu) \indmn$.
We show that if $|t_\lambda| h = \sum_{\bS \in \NS_\col} c_\bS \ecol(\bS)$
then there is a conjugate-semistandard tableau family tuple of shape $\mu^\kappa$
and type~$\lambda$ corresponding to each $\bS$ with $c_\bS \not= 0$.
If $m$ is odd we obtain the same result by replacing $\widetilde{N}/\widetilde{R}$
with $N/R$ and using the submodule with basis~$\mathcal{B}_\row$.
To make this correspondence precise we require the definitions below. Recall
from~\S\ref{sec:nonzero} that if $t$ is a $\mu$-tableau with entries
from $\Omega^\lambda$ then the deindexed tableau $D(t)$ is the $\mu$-tableau with entries from $\N$ obtained
by removing indices from the entries~of~$t$. 

\begin{definition}
Let $\bT \in \NOmega(\Omega^\lambda)$.
\begin{defnlist}
\item Suppose $m$ is even. For each $j \in \{1,\ldots, \nu_1\}$,
let $\mathcal{T}_j = \{D(t_1), \ldots, D(t_{\nu_j'})\}$ where $t_1,\ldots,t_{\nu_j'}$
are the entries in column $j$ of $\bT$.
Define $D_\col(\bT) = (\mathcal{T}_1,\ldots,\mathcal{T}_{\nu_1})$.

\item Suppose $m$ is odd. For each $i \in \{1,\ldots, \nu'_1\}$,
let $\mathcal{T}_i = \{D(t_1), \ldots, D(t_{\nu_i})\}$ where $t_1,\ldots,t_{\nu_i}$
are the entries in row $i$ of $\bT$.
Define $D_\row(\bT) = (\mathcal{T}_1,\ldots,\mathcal{T}_{\nu'_1})$.
\end{defnlist}
\end{definition}

\begin{definition}\label{def:sep}
Let $t$ be a $\mu$-tableau with entries from $\Omega^\lambda$. We say that
$t$ is \emph{separated} if no two distinct symbols $y_x$ and $y_{x'}$ lie in the
same row of $t$.
\end{definition}

Theorem~\ref{thm:toFamily} is an easy corollary of the
following proposition and its analogue for $m$ odd.

\begin{proposition}\label{prop:toFamily}
Suppose $m$ is even.
Let $h : \widetilde{M}^\lambda \rightarrow (S^\mu \oslash \widetilde{M}^\nu) \indmn$
be a homomorphism of $\Q S_{mn}$-modules. Identifying the domain
with the $\Q S_{\Omega^\lambda}$-module generated by $|t_\lambda|$ and the
 codomain with the submodule
of $\widetilde{N}/\widetilde{R}$ with basis $\mathcal{B}_\col$, let
\begin{equation}\label{eq:th} |t_\lambda|h = \sum_{\bU \in \R} c_\bU \ecol(\bU) \end{equation}
where $\R \subseteq \NS_\col$ and $c_\bU\not = 0$ for each $\bU \in \R$.
\begin{thmlist}
\item If $\bU \in \R$ and $t$ is a $\mu$-tableau entry of $\bU$ then $t$ is separated.
\item If $\bU \in \R$ and $s$ and $t$ are  $\mu$-tableau entries in different positions
in the same column of $\bU$
then
$D(s) \not= D(t)$.
\end{thmlist}
\end{proposition}



For completeness, we state the version of the proposition for $m$ odd.

\begin{proposition}\label{prop:toFamilyOdd}
Suppose $m$ is odd.
Let $h : \widetilde{M}^\lambda \rightarrow (S^\mu \oslash M^\nu) \indmn$
be a homomorphism of $\Q S_{mn}$-modules. 
Identifying the domain
with the $\Q S_{\Omega^\lambda}$-module generated by $|t_\lambda|$ and 
 the
 codomain with the submodule
of $ N/ R$ with basis $\mathcal{B}_\row$, let
\begin{equation}\label{eq:th2} |t_\lambda|h = \sum_{\bU \in \R} c_\bU \erow(\bU) \end{equation}
where $\R \subseteq \NS_\row$ and $c_\bU\not = 0$ for each $\bU \in \R$. 
\begin{thmlist}
\item If $\bU \in \R$ and $t$ is a $\mu$-tableau entry of $\bU$ then $t$ is separated.
\item If $\bU \in \R$ and $s$ and $t$ are  $\mu$-tableau entries in different positions
in the same row of $\bU$
then $D(s) \not= D(t)$.
\end{thmlist}
\end{proposition}

%

\begin{proof}[of Theorem~\ref{thm:toFamily} assuming Propositions~\ref{prop:toFamily} and~\ref{prop:toFamilyOdd}]
We first suppose that $m$ is even, then indicate the very minor changes needed if $m$ is odd.
Pull back a non-zero homomorphism
$\smash{S^\lambda \rightarrow (S^\mu \oslash S^\nu) \indmn}$ to 
a non-zero homomorphism
$\smash{\widetilde{M}^\lambda \rightarrow (S^\mu \oslash \widetilde{M}^\nu) \indmn}$,
and let $h : \widetilde{M}^\lambda \rightarrow \widetilde{N}/\widetilde{R}$ 
be the corresponding homomorphism of $\Q S_{\Omega^\lambda}$-modules
with image contained in the submodule of $\widetilde{N}/\widetilde{R}$
with basis $\{ \ecol(\bS) : \bS \in \NS_\col \}$. Let $|t_\lambda|h 
 = \sum_{\bS \in \R} c_\bS \ecol(\bS)$, as in~\eqref{eq:th}.

Let $\bT \in \R$.
If $t$ is a $\mu$-tableau entry of
$\bT$, then $t$ is standard, so if $y_x, y'_{x'}$ are entries in the same
row of $t$ then $y_x < y'_{x'}$. Moreover, by Proposition~\ref{prop:toFamily}(i),
we have $y\not=y'$. Therefore the rows of $D(t)$ are strictly increasing.
Similarly, since $t$ is standard, the columns of $D(t)$ are weakly increasing.
It now follows from Proposition~\ref{prop:toFamily}(ii) that
$D_\col(\bT)$ is a conjugate-semistandard tableau family tuple
of shape $\mu^\kappa$. (Recall that $\kappa = \nu'$ when $m$ is even.)

Suppose that  $\bU \in \NS_\col$ is such that
 $D_\col(\bU) = D_\col(\bT)$.
 There is an index permutation
$\theta \in I(\lambda)$ and a column place permutation 
$\tau \in \CPP(\nu)$ such that $\bU = \bT \theta \cdot \tau $.
Observe that $e_\col(\bU) = e_\col(\bT \theta \cdot \tau) = \sgn(\tau) e_\col(\bT \theta)$. 
Applying $\theta$ to both sides of~\eqref{eq:th}, we see from the
summand $c_\bT e_\col(\bT) \theta$ that the coefficient of 
$e_\col(\bU)$ in $|t_\lambda|\theta h$ is $\sgn(\tau) c_\bT$. Since
$|t_\lambda|\theta$ = $\sgn(\theta)|t_\lambda|$, it follows
that $\sgn(\theta) c_\bU = \sgn(\tau) c_\bT$.
Hence the number of independent choices for the coefficients in~\eqref{eq:th}
is at most~$q$, the number of conjugate-semistandard tableau family tuples of shape $\mu^{\kappa}$
and type~$\lambda$.
%
%
Therefore
\hbox{$\dim \Hom_{\Q S_\Omega} (\widetilde{M}^\lambda, \langle \mathcal{B}_\col \rangle) \le q$}.
Since
\[ \dim \Hom_{\Q S_{mn}} \bigl( \widetilde{M}^\lambda, (S^\mu \oslash S^\nu) \Indmn \bigr) \le
\dim \Hom_{Q \Omega^\lambda} (\widetilde{M}^\lambda, \langle \mathcal{B}_\col \rangle) \]
the theorem follows.

For $m$  odd, instead start with a non-zero homomorphism $S^\lambda \rightarrow
(S^\mu \oslash \widetilde{S}^\nu) \indmn$ and pull it back
to a non-zero homomorphism $\widetilde{M}^\lambda \rightarrow (S^\mu
\oslash \widetilde{M}^\nu)\indmn$ 
so that Proposition~\ref{prop:toFamilyOdd} may be used.
 Then replace $\ecol(\bS) \in \NS_\col$ with $\erow(\bS) \in \NS_\row$,
$\mathcal{B}_\col$ with $\mathcal{B}_\row$ 
and $D_\col(\bT)$ with $\D_\row(\bT)$ throughout, and use that  $\erow(\bT\theta \cdot \tau)=  
\erow(\bT)\theta$ for $\tau \in \RPP(\nu)$ and $\theta \in I(\lambda)$. 
\end{proof}

\subsection{Proof of Proposition~\ref{prop:toFamily}}

The main result needed to 
prove Proposition~\ref{prop:toFamily} 
(and also the analogous Proposition~\ref{prop:toFamilyOdd}) 
is Proposition~\ref{prop:Pieri}
below; this result may be seen as a modular version of Pieri's rule.  
An example is given following the proof. 
Recall that a subset of a Young diagram is said to be a \emph{vertical strip}
if it contains no two boxes in the same row.

\begin{proposition}\label{prop:Pieri}
Let $\Gamma$ be a subset of $\Omega^\lambda$ of size $m$. Let
\[ b = \sum_{\theta \in S_\Gamma \cap I(\lambda)} \theta \sgn(\theta). \]
Let $d = \lambda_1$.
For each $y \in \{1,\ldots, d\}$, let $\Gamma_y = \{ y_1, \ldots, y_{\lambda_y'} \} \cap \Gamma$.
The elements $e(t)b$ such that
\begin{thmlist}
\item $t$ is a standard $\mu$-tableau with entries from $\Gamma$,
\item the boxes of $t$ occupied by the symbols in each~$\Gamma_y$ form
a vertical strip, 
\item if $y_x$, $y_{x'} \in \Gamma_y$ and $x < x'$ then
$y_x$ appears in a lower numbered row than $y_{x'}$,
\end{thmlist}
are a basis for the $\Q S_\Gamma$-module $S^\mu b$.
\end{proposition}

\begin{proof}
By repeated applications of 
Pieri's rule (in the dual form stated in \cite[page 340]{StanleyII}), the multiplicity of $\chi^\mu$ in 
$\bigl(\prod_{y=1}^d \sgn_{\Gamma_y} \bigr)\Ind^{S_\Gamma}$ is the number of 
sequences of partitions
\begin{equation} 
\varnothing = \mu_0 \subset \mu_1 \subset \ldots \subset \mu_{d-1} \subset \mu_d = \mu \label{eq:vbs}
\end{equation}
such that each $[\mu_y] / [\mu_{y-1}]$ is a vertical strip and $|\mu_y| - |\mu_{y-1}| = |\Gamma_y|$.
Let~$S$ be the set of such sequences. 
Define $\rule{0pt}{12pt}\widehat{b} = b \,/ \prod_{y=1}^d |\Gamma_y|$. Note that $\widehat{b}$ is an idempotent and
that if $v \in S^\mu \widehat{b}$ then $\langle v \rangle$ affords the sign representation
of $\prod_{y=1}^d \Gamma_y$. By Frobenius reciprocity, we see that $\dim S^\mu \widehat{b} = |S|$.

Given a sequence in $S$ as in~\eqref{eq:vbs}, let $t$ be the standard $\mu$-tableau with entries from $\Gamma$
such that, for each $y$, 
the boxes $[\mu_y] / [\mu_{y-1}]$ are occupied by the elements of $\Gamma_y$, ordered as in condition
(iii).
Let $T$ be the set of tableaux obtained in this way.
Notice that if $t \in T$ and $\theta \in S_\Gamma \cap I(\lambda)$ then the tableau~$u$ obtained
from $t \theta$ by sorting its columns into increasing order is standard. 
Let $\theta_u$ be the unique permutation such that \hbox{$t \theta \theta_u = u$}.
If $\theta' \in S_\Gamma \cap I(\lambda)$ gives the same tableau $u$
after sorting the columns of $t\theta'$, then
$\theta \theta_u = \theta' \theta'_u$, so $\sgn(\theta) \sgn(\theta_u) = \sgn(\theta')\sgn(\theta'_u)$.
Hence  there exists $C \in \N$ such that the expansion
of $e(t) b$ in the standard basis of $S^\mu$ is 
\[ e(t) b = C \sum_{u} c_u e(u), \]
where the sum is over all standard $\mu$-tableaux $u$ with entries from $\Gamma$
such that the sets of entries
of $u$ and $t$ in each vertical strip $[\mu_y] / [\mu_{y-1}]$ are equal, and
$c_u \in \{ +1, -1\}$ for each $u$. Hence
the $e(t) b$ for $t \in T$ are linearly independent. Since $\dim S^\mu b = \dim S^\mu \widehat{b} 
= |S| = |T|$,
they form a basis for~$S^\mu b$.
\end{proof}

For example, if $\Gamma = \Omega^{(3,2)} = \{1_1,1_2,2_1,2_2,3_1\}$, 
then
\[ b_{(3,2)} = \sum_{\theta \in S_{\Gamma} \cap I(\lambda)} \theta \sgn(\theta) = 
\bigl(1 - (\oa,\ob) \bigr)\bigl(1-(\ta,\tb)\bigr),\] 
and
\[
\makeatletter
\def\y@vr{\vrule height0.7\y@b@xdim width\y@linethick depth 0.3\y@b@xdim}
\def\y@setdim{%
  \ify@autoscale%
   \ifvoid1\else\typeout{Package youngtab: box1 not free! Expect an
     error!}\fi%
   \setbox1=\hbox{A}\y@b@xdim=1.8\ht1 \setbox1=\hbox{}\box1%
  \else\y@b@xdim=\y@boxdim \advance\y@b@xdim by -2\y@linethick
  \fi}
\makeatother
 \setlength{\arrayrulewidth}{.08em}
 S^{(3,2)}b_{(3,2)} = \Bigl< e \Bigl(\, \young(\oa\ta\da,\ob\tb)\, \Bigr) \Bigr>. \]
If instead $\mu = (2,2,1)$ then
\[
\makeatletter
\def\y@vr{\vrule height0.7\y@b@xdim width\y@linethick depth 0.3\y@b@xdim}
\def\y@setdim{%
  \ify@autoscale%
   \ifvoid1\else\typeout{Package youngtab: box1 not free! Expect an
     error!}\fi%
   \setbox1=\hbox{A}\y@b@xdim=1.8\ht1 \setbox1=\hbox{}\box1%
  \else\y@b@xdim=\y@boxdim \advance\y@b@xdim by -2\y@linethick
  \fi}
\makeatother
 \setlength{\arrayrulewidth}{.08em}
 S^{(2,2,1)}b_{(3,2)} = \Bigl< e\Bigl(\, \young(\oa\ta,\ob\tb,\da)\, \Bigr),
\  e\Bigl(\, \young(\oa\ta,\ob\da,\tb)\, \Bigr) -  e\Bigl(\, \young(\oa\tb,\ob\da,\ta) \, \Bigr) \Bigr>. \]
The first two tableaux above form the set $T$ defined in the proof; correspondingly,
Pieri's rule implies the multiplicity of the 
character of $S^{(2,2,1)}$ in $(\sgn_{S_2} \times \sgn_{S_2})\ind^{S_5}$ is $2$.

We remark that using the James--Peel filtration of a Specht module by Specht modules
labelled by skew-partitions (see \cite[\S 3]{JamesPeel}) 
one can adapt the proof of Proposition~\ref{prop:Pieri} to avoid assuming
Pieri's rule; the rule then follows as a corollary.

\begin{proof}[of Proposition~\ref{prop:toFamily}]
We have $|t_\lambda| b_\lambda = C |t_\lambda|$ where $C = |I_\lambda|$.
Hence the image of the homomorphism $h$ is contained in the subspace of $\rule{0pt}{11pt}
\widetilde{N}/\widetilde{R}$ 
spanned by the $e_\col(\bT) b_\lambda$ for $\bT \in \NS_\col$.

Let $\bT \in \NS_\col$. For each $(i,j) \in [\nu]$, let $\Gamma_{(i,j)}$ be the
set of entries of the $\mu$-tableau $(i,j)\bT$, let $I(\lambda)_{(i,j)} =
S_{\Gamma_{(i,j)}} \cap I(\lambda)$ and let
\[ b_{(i,j)} = \sum_{\theta \in S_{I(\lambda)_{(i,j)}}}\! \theta \sgn(\theta).\] 
Choosing coset representatives $\phi_1, \ldots, \phi_e$ 
for $\prod_{(i,j) \in [\nu]} I(\lambda)_{(i,j)}$ in $I(\lambda)$, we see that
\[ e_\col(\bT) b_\lambda = \sum_{k=1}^e e_\col(\bT) \prod_{(i,j) \in [\nu]} b_{(i,j)} \phi_k \sgn(\phi_k). \]
Each $b_{(i,j)}$ acts on $\bT$ by permuting the entries of the $\mu$-tableau
in position $(i,j)$ of $\bT$, so by Proposition~\ref{prop:Pieri}, 
$e_\col(\bT) \prod_{(i,j)} b_{(i,j)}$ is a linear combination of $e_\col(\bU)$ for $\bU \in \NS_\col$
such that each $\mu$-tableau entry of $\bU$ is separated.
Since the coset representatives $\phi_k$ permute indices while leaving numbers
fixed, we have $e_\col(\bT) b_\lambda = \sum_{\bU \in \NS_\col} c_\bU e_\col(\bU) $
where if $c_\bU \not=0$ then each $\mu$-tableau entry of $\bU$ is separated.
It follows that
\[ \langle e_\col(\bT) b_\lambda  : \bT \in \NS_\col \rangle \subseteq W \]
where $W$ is the subspace of $\widetilde{N}/\widetilde{R}$ spanned by the
$e_\col(\bU)$ for $\bU \in \NS_\col$ such that each $\mu$-tableau entry of $\bU$ is separated.

To show that condition (ii) in Proposition~\ref{prop:toFamily} holds, it will
be convenient to say that
$\bT \in \NS_\col$ has a \emph{column repeat}
if there exist distinct
$(i,j)$ and $(i',j) \in [\nu]$ such
that the $\mu$-tableau entries $s = (i,j)\bT$ and $t = (i',j)\bT$ of $\bT$
satisfy $D(s) = D(t)$.
Suppose that $\bT \in \NS_\col$ has a column repeat, described by this notation, and let
\[ \psi = \prod_{(a,b) \in [\mu]} \bigl((a,b)s, (a,b)t \bigr). \]
By hypothesis, $\psi \in I(\lambda)$.
Observe that $s \psi = t$, $t\psi = s$, and that $\psi$ leaves all other
$\mu$-tableau entries of $\bT$ fixed. 
Therefore $|\bT \psi| = -|\bT|$. More generally, if $\pi \in \CPP(\mu)^\timesn$ then,
since the index permuting
action of $I(\lambda)$ commutes with place permutations,  we have
\[ |\bT \star \pi| (1 + \psi) = |\bT \star \pi| + |\bT \psi \star \pi| = 
|\bT \star \pi| - |\bT \star \pi' | \]
where $\pi' \in \CPP(\mu)^\timesn$ is obtained from $\pi$ by swapping
the permutations indexed by $(i,j)$ and $(i',j)$. Hence,
taking coset representatives $\theta_1, \ldots, \theta_d$ for $\langle \psi \rangle \le I(\lambda)$,
and noting that $\sgn(\psi) = 1$ since $m$ is even, we get
\begin{align*} e_\col(\bT) b_\lambda &= 
\!\!\!\sum_{\pi \in \CPP(\mu)^\timesn} \!\! |\bT \star \pi| (1 + \psi) \sum_{k=1}^d \theta_k \sgn(\theta_k)  \\
&= \!\!\!\sum_{\pi \in \CPP(\mu)^\timesn} \!\!\bigl( |\bT \star \pi| - |\bT \star \pi'| \bigr)
 \sum_{k=1}^d \theta_k \sgn(\theta_k)  \\
&= \!\!\!\sum_{\pi \in \CPP(\mu)^\timesn} \!\! \bigl( |\bT \star \pi| \bigr) \sum_{k=1}^d \theta_k \sgn(\theta_k) 
-  \!\!\!
\sum_{\pi \in \CPP(\mu)^\timesn} \!\! \bigl( |\bT \star \pi| \bigr) \sum_{k=1}^d \theta_k \sgn(\theta_k)\\
&= 0.
\end{align*}
It therefore suffices to show that if $\bT$ has separated $\mu$-tableau
entries, $\bT$ does not have a column repeat
and $\e_\col(\bU)$ appears in the expansion of $e_\col(\bT) b_\lambda$ 
in the basis $\NS_\col$, then $\bU$ does not have a column repeat.

Let $\theta \in I(\lambda)$. 
For each $(i,j) \in [\nu]$, let $v_{(i,j)}$ be the $\mu$-tableau obtained from
$(i,j)\bT \theta$ by sorting its columns into increasing order. Since $(i,j)\bT$ is a
separated $\mu$-tableau,
each $v_{(i,j)}$ is standard. Let $\bV$ be defined by $(i,j)\bV = v_{(i,j)}$. 
Sorting the columns of $\bV$ into increasing order in the order $<_\i$ 
we obtain $\bU \in \NS_\col$ such that $e_\col(\bT) \theta =  \pm e_\col(\bU)$.
Since $\theta$ permutes indices while leaving numbers fixed, we have
\[ \bigl\{ D\bigl( (1,j)\bT \bigr), \ldots, D\bigl( (\nu_j', j) \bT \bigr)  \bigr\}
= \bigl\{ D\bigl( (1,j)\bU \bigr), \ldots, D\bigl( (\nu_j', j) \bU \bigr) \bigr\} \]
for each $j \in \{1,\ldots,\nu_1\}$.
Therefore $\bU$ does not have a column repeat, as required.
\end{proof}

\begin{proof}[of Proposition~\ref{prop:toFamilyOdd}]
If $m$ is odd then the proof should be modified by replacing columns with rows.
In the second step of the proof we have 
$\{\bT\} \psi = \{\bT\}$ and $\sgn(\psi) = -1$, since $m$ is odd,
 and a similar argument shows that $e_\row(\bT)b_\lambda = 0$.
\end{proof}

\section{Proofs of Theorems~\ref{thm:mainMaximal} and~\ref{thm:main}}
\label{sec:proof}

Theorem~\ref{thm:main}
follows easily  using the following  lemma which summarises the salient points from  Proposition~\ref{prop:nonzero} and Theorem~\ref{thm:toFamily}.

\vbox{
\begin{lemma}{\ }
\begin{thmlist}
\item If $S^\lambda$ is a summand of $H_\mu^\nu$ then there is a conjugate-semistandard 
tableau family tuple of shape $\mu^\kappa$ and type $\lambda^\star$
with $\lambda \unrhd \lambda^\star$.

\item If there is a conjugate-semistandard tableau family tuple of shape $\mu^\kappa$
and type $\lambda^\star$ then there is a summand $S^{\lambda^\dagger}$
of $H_\mu^\nu$ with $\lambda^\star \unrhd \lambda^\dagger$.
\end{thmlist}
\end{lemma}}

\begin{proof}
Using Theorem~\ref{thm:toFamily}, the hypothesis of part (i) implies that there is at least one 
conjugate-semistandard  tableau family tuple of shape $\mu^\kappa$ and type $\lambda$. We may therefore take $\lambda^\star$ to equal $\lambda$ here and the conclusion holds.
For (ii), by Proposition~\ref{prop:nonzero} there is a non-zero homomorphism
from $\widetilde{M}^{\lambda^\star}$ into $H_\mu^\nu$.
The result now follows
since, by Lemma~\ref{lemma:Mtilde}, if $S^{\lambda^\dagger}$ is a summand of $\widetilde{M}^{\lambda^\star}$
then $\lambda^\star \unrhd \lambda^\dagger$.
\end{proof}

This completes the proof of Theorem~\ref{thm:main}.
To deduce Theorem~\ref{thm:mainMaximal} we twist by the sign representation.
The restriction of $\sgn_{S_{mn}}$ to $S_m \wr S_n$ is
$\sgn_{S_m} \oslash\,\, \Q_{S_n}$ if $m$ is even and $\sgn_{S_m} \oslash \sgn_{S_n}$ if $m$ 
is odd. Hence
\[ H_\mu^\nu\, \otimes\, \sgn_{S_{mn}} \!\! = (S^\mu \oslash S^\nu) \indmn 
\otimes \sgn_{S_{mn}} \!\cong \!
\begin{cases}
(S^{\mu'} \oslash S^\nu)\indmn &\! \text{if $m$ is even} \\
(S^{\mu'} \oslash S^{\nu'})\indmn & \!\text{if $m$ is odd.} 
\end{cases} \]
It now follows from~\eqref{eq:sgntwist} in \S\ref{subsec:Specht} that
$S^\lambda$ is a maximal summand of $H_\mu^\nu$ if and only if $S^{\lambda'}$ is a minimal summand of $H_{\mu'}^\nu$, if $m$ is even, or $H_{\mu'}^{\nu'}$, if $m$ is odd. By Theorem~\ref{thm:main} this 
holds (in either case)
 if and only if $\lambda'$ is a minimal type, that is $\lambda$ is maximal among the partitions that occur as a weight  of a conjugate-semistandard tableau family tuple of shape $\eta^{\nu'}$ where $\eta = \mu'$.

We remark that, using Proposition~\ref{prop:type} (whose proof was not included in this article) and Lemma~\ref{lemma:minimalimpliesclosed}, the statement of Theorem~\ref{thm:mainMaximal} may be strengthened.  We may remove the restriction that  $\lambda$ is maximal \emph{among the partitions} that occur as a weight and conclude instead that $\lambda$ is simply a maximal weight of the conjugate-semistandard tableau family tuple.

\section{Linear independence of homomorphisms} \label{sec:indep}
In this section we  extend the argument in the proof of Proposition~\ref{prop:nonzero} to give a sufficient condition for
homomorphisms defined using different conjugate-semistandard
tableau family tuples of the same type to be linearly independent.

 To simplify the statement of this result we write $t \in \mathcal{T}$ to mean
that the tableau $t$ belongs to one of the tableau families in the conjugate-semistandard 
tableau family tuple $\mathcal{T}$.
Let $\preceq_\col$ denote the order on $\mu$-columnar tabloids
defined in the proof of Lemma~\ref{lemma:RPP}, obtained 
by taking $\preceq$ to be the usual order on $\N$ in Definition~\ref{defn:preorder}.
Recall that
$\kappa = \nu'$ if $m$ is even and $\kappa = \nu$
if $m$ is odd.

\begin{proposition}\label{prop:linindep}
Let $\mathcal{T}^{(1)}, \ldots, \mathcal{T}^{(d)}$ be conjugate-semistandard
tableau family tuples of shape $\mu^\kappa$ and type $\lambda$.
Suppose that for each $e \in \{1,\ldots, d\}$ there exists a conjugate-semistandard
tableau $s^{(e)} \in \mathcal{T}^{(e)}$ such that 
if $c < e$ and $u \in \mathcal{T}^{(c)}$ then $||u|| \prec_\col ||s^{(e)}||$.
\begin{thmlist}
\item If $m$ is even then 
the homomorphisms
$\overline{f}_{\mathcal{T}^{(1)}}, \ldots, \overline{f}_{\mathcal{T}^{(d)}}
: S^\lambda \rightarrow H_\mu^\nu$ are linearly independent. 
\item If $m$ is odd then
the homomorphisms
$\overline{g}_{\mathcal{T}^{(1)}}, \ldots, \overline{g}_{\mathcal{T}^{(d)}}
: S^\lambda \rightarrow H_\mu^\nu$ are linearly independent. 
\end{thmlist}
\end{proposition}

\begin{proof}
We give the proof for $m$ even and then explain the minor changes needed if $m$ is odd.

For each $c \in \{1,\ldots, d\}$, let $\bT^{(c)}$ be the tableau in $\NOmega(\Omega^\lambda)$ 
corresponding to~$\mathcal{T}^{(c)}$. 
Suppose for a contradiction that there is a linear dependency involving
$\overline{f}_{\mathcal{T}^{(1)}}, \ldots, \overline{f}_{\mathcal{T}^{(e)}}$
in which the coefficient of $\overline{f}_{\mathcal{T}^{(e)}}$ is non-zero.
The proof of Proposition~\ref{prop:nonzero} shows that the coefficient of $\{\bT^{(e)}\} + C$
in $|t_\lambda|\overline{f}_{\mathcal{T}^{(e)}}$ is non-zero. 
Using the definition of $\overline{f}_{\mathcal{T}^{(c)}}$ in Proposition~\ref{prop:homsExplicit}, we see that there exists $c < e$ 
and $\theta \in I(\lambda)$, $\tau \in \CPP(\nu)$ and $\sigma \in \RPP(\mu)^n$
such that
$\{ \bT^{(e)} \} + C = \pm \{ (\bT^{(c)} \theta \cdot \tau) \star \sigma \} + C$.
Hence there exist $\pi \in \CPP(\mu)^n$ and $\rho \in \RPP(\nu)$ such that
\[ \bT^{(e)} = (((\bT^{(c)} \theta \cdot \tau) \star \sigma)\star \pi)\cdot \rho. \]
Let $t_{(i,j)} = (i,j)\bT^{(e)}$ and $u_{(i,j)} = (i,j)\bT^{(c)}$
for $(i,j) \in [\nu]$.

Repeating the argument in the first step
in the proof of Proposition~\ref{prop:nonzero}, we get
\[ \{ ||D(t_{(i,j)})|| : (i,j) \in [\nu] \} = \{||D(u_{(i,j)} \cdot \sigma_{(i,j)\tau})|| :
(i,j) \in [\nu]\}. \]
But by Lemma~\ref{lemma:preorder}(i), we have
$||D(u_{(i,j)}  \cdot \sigma_{(i,j)\tau})|| \preceq_\col ||D(u_{(i,j)})|| \prec_\col ||D(s^{(e)})||$.
By transitivity, $||D(u_{(i,j)} \cdot \sigma_{(i,j)\tau})|| \prec_\col ||D(s^{(e)})||$; in particular,
$||D(u_{(i,j)}) \cdot \sigma_{(i,j)\tau})|| \not= ||D(s^{(e)})||$.
Since $||D(s^{(e)})||$ is a member of the left-hand set, this is a contradiction.

If $m$ is odd then instead we look at the coefficient
of $|\bT^{(e)}| + \widetilde{C}$, and 
instead take $\tau \in \RPP(\nu)$ and $\rho \in \CPP(\nu)$. The proof is otherwise unchanged.
\end{proof}

\begin{example}\label{ex:excont}
Let 
\[ s_1 = \young(12,1)\,,\ s_2 = \young(12,2)\,,\ s_3 = \young(12,3)\,,\ s_4 = \young(13,1)\,,\ s_5 = \young(13,2)\,. \]
Define $\mathcal{S} = \bigl( \{s_1,s_2,s_3,s_4\}, \{s_1\} \bigr)$.
Note that $\mathcal{S}$ has the same shape and type as $\mathcal{T}
= \bigl( \{s_1,s_2,s_4,s_5\},\{s_1\} \bigr)$ from Example~\ref{ex:ex}.
Since $||s_5||$ is the greatest element of $\{||s_1||,\ldots, ||s_5||\}$
in the order $\preceq_\col$ used in Proposition~\ref{prop:linindep},
the homomorphisms $\overline{g}_\mathcal{S}$ and $\overline{g}_\mathcal{T}$
are linearly independent. Hence
the multiplicity of $S^{(3^2,2^3,1^3)}$ in $H_{(2,1)}^{(4,1)}$ is at least $2$.

If $(\mathcal{U}_1, \mathcal{U}_2)$ is a conjugate-semistandard tableau family tuple
of shape $(2,1)^{(4,1)}$ of minimal type then, by Lemma~\ref{lemma:minimalimpliesclosed},
$\mathcal{U}_1$ and $\mathcal{U}_2$ are closed. Hence $\mathcal{U}_2 = \{ s_1\}$ and
$\mathcal{U}_1$ appears in the table in Figure~2. 
 Therefore $\mathcal{S}$ and $\mathcal{U}$
are the only conjugate-semistandard tableau family tuples of shape $(2,1)^{(4,1)}$
and type $(3^2,2^3,1^3)$. Thus
Theorem~\ref{thm:toFamily} implies that
the multiplicity of $S^{(3^2,2^3,1^3)}$
in $H_{(2,1)}^{(4,1)}$ is at most $2$.
Therefore the multiplicity is exactly~$2$.
\end{example}

To motivate a further example, we remark on one obvious source of linearly dependent homomorphisms.
Suppose that $n$ is even and that
$\mathcal{U}$ and~$\mathcal{V}$ are conjugate-semistandard tableau families  of shape $\mu^{r}$
where $r = n/2$.
Let~$\lambda$ be the type of 
$(\mathcal{U}, \mathcal{V})$ 
and let $\bT_{(\mathcal{U},\mathcal{V})}$ and $\bT_{(\mathcal{V}, \mathcal{U})}$ 
be the corresponding elements of
$\NOmega(\Omega^\lambda)$.
By the definition in \S\ref{sec:homs}, $\bT_{(\mathcal{V},\mathcal{U})}$ 
is obtained from 
$\bT_{(\mathcal{U}, \mathcal{V})}$ by swapping the two
$\mu$-tableaux in each row, if $m$ is even, and in each column,
if $m$ is odd.
It now follows easily from Proposition~\ref{prop:homsExplicit} that when $m$ is even the
 homomorphisms corresponding to $(\mathcal{U}, \mathcal{V})$ and $(\mathcal{V}, \mathcal{U})$
 from $S^\lambda$ into $\smash{(M^\mu \oslash S^{(2^{r})})\indmn}$ are equal.
 If $m$ is odd the homomorphisms from $S^\lambda$ into $(M^\mu \oslash \widetilde{S}^{(r,r)})\indmn$ 
agree up to a sign of $(-1)^{r}$. In either case, they are linearly dependent.

\begin{example}\label{ex:big}
Define
\[ \exu = \young(12,4)\,,\ \exv = \young(23,2)\,,\ \exw = \young(13,3) \,,\  \exx = \young(14,2)\, . \]
Note that these tableaux are incomparable in the majorization order.
The set of conjugate-semistandard tableaux majorized by one of $\exu$, $\exv$, $\exw$ and $\exx$ is
\[
\mathcal{T} = \Bigl\{ \, \sssyoung{(12,1)}\,,
\,\sssyoung{(12,2)}\,,\,\sssyoung{(12,3)}\,,\,\sssyoung{(12,4)}\,,\,
\sssyoung{(13,1)}\,,\,\sssyoung{(13,2)}\,,\, \sssyoung{(23,2)}\,,\,\sssyoung{(13,3)}\,,\, \sssyoung{(14,1)}\,,
\,\sssyoung{(14,2)}\, \Bigr\}.\]
Let $t_1, \ldots, t_{10}$ denote these tableaux, in increasing order under $<$, the total order from 
\S\ref{subsec:totalOrder}, as written above.
Note that $\exu = t_4$, $\exv = t_7$, $\exw= t_8$ and $\exx = t_{10}$.
The type $\lambda = (4^4,3^5,2^5,1^7)$ is minimal for the existence of
a conjugate-semistandard tableau family tuple of shape $(2,1)^{(8,8)}$;
up to the equivalence $(\mathcal{V}, \mathcal{U}) \sim (\mathcal{U},\mathcal{V})$ there
are exactly five such tuples. 
Writing $\mathcal{S}_{ab}=\{t_1,t_2,t_3,t_5,t_6,t_9\} \cup \{a,b\}$ for $a, b$ distinct elements of $\{\exu,\exv,\exw,\exx\}$, the five tuples are
\[ (\mathcal{S}_{\exu\exv},\mathcal{S}_{\exw\exx}),(\mathcal{S}_{\exu\exx},\mathcal{S}_{\exw\exv}),
(\mathcal{S}_{\exu\exw},\mathcal{S}_{\exv\exx}), (\mathcal{S}_{\exu\exv},\mathcal{S}_{\exu\exw}),(\mathcal{S}_{\exv\exx},\mathcal{S}_{\exw\exx}).\]
(This claim is not logically essential to this example; it may be verified using 
the Haskell \cite{Haskell98} program \texttt{TableauFamilies}
available from the second author's website\footnote{\url{www.ma.rhul.ac.uk/~uvah099}}.)   

For each $i \in \{1,2,3,4,5\}$, let
$\mathbf{S}_i \in \NOmega(\Omega^\lambda)$ be the  $(8,8)$-tableau with $(2,1)$-tableau entries
corresponding to the $i$th tuple above, 
as defined in \S\ref{subsec:tableaux}. 
Let $\mathbf{S}^\dagger_i \in \NOmega(\Omega^\lambda)$ be obtained from $\mathbf{S}_i$ by 
reordering the entries in each of its two rows
so that the two $(2,1)$-tableaux with entries from $\Omega^\lambda$ obtained by appending
indices to 
elements of $\{\exu,\exv,\exw,\exx\}$  
appear in positions $(1,7), (1,8),
(2,7), (2,8)$. We may perform this reordering so
 that applying the deindexing map $D$ to the 
entries in positions $(1,7)$, $(1,8)$, $(2,7)$, $(2,8)$ 
of the~$\mathbf{S}^\dagger_i$ gives the following tableaux:
\begin{equation}\label{eq:exts} 
\young(\exu\exv,\exw\exx)\, , \ \, \young(\exu\exx,\exw\exv)\, , \ \, \young(\exu\exw,\exv\exx) \, , \
\young(\exu\exv,\exu\exw)\, , \ \, \young(\exv\exx,\exw\exx)\, . \end{equation}
For each $i$ we have a corresponding homomorphism
$g_i : \widetilde{M}^\lambda \rightarrow (M^{(2,1)} \oslash M^{(8,8)}) \indmne$, as defined in~\S\ref{subsec:homs}.
Since $\lambda$ has minimal type, $g_i$ induces a non-zero homomorphism $S^\lambda 
\rightarrow (M^{(2,1)} \oslash M^{(8,8)}) \indmne$.
Define a homomorphism
$\smash{\tilde{g_i} : S^\lambda \rightarrow (M^{(2,1)} \oslash \widetilde{S}^{(8,8)}) \indmne}$ 
by composing this induced map with the canonical surjection 
$(M^{(2,1)} \oslash M^{(8,8)}) \indmne \rightarrow (M^{(2,1)} \oslash \widetilde{S}^{(8,8)})\indmne$. 
Thus
\[ 
e(t_\lambda)\widetilde{g}_i 
= \!\!\!\!\sum_{\theta \in I(\lambda) \atop \tau \in \RPP((8,8))} \!\!\!|\mathbf{S}_i \theta \cdot \tau| \sgn(\theta)
+ \widetilde{R} 
=
\!\!\!\sum_{\theta \in I(\lambda) \atop \tau \in \RPP((8,8))} \!\!\!
|\mathbf{S}_i^\dagger \cdot \tau| \theta \sgn(\theta)
+ \widetilde{R} \]
for each $i$, where the second equality holds because $\mathbf{S}_i = \mathbf{S}_i^\dagger\sigma_i$ 
for some $\sigma_i \in \RPP\bigl( (8,8) \bigr)$, and 
the place permutation action of
$\RPP\bigl((8,8)\bigr)$ commutes with the index permutation action of $I(\lambda)$.
Hence
\begin{equation}\label{eq:exS1} e(t_\lambda) 
\widetilde{g}_i = 
\sum_{\theta \in I(\lambda)} \widetilde{e}(\mathbf{S}_i^\dagger)\theta\sgn(\theta) + \widetilde{R}
\end{equation}
for each $i$.
Indicating $(2,1)$-tableaux with entries from $\Omega^\lambda$
by the deindexed tableaux used to define them,~\eqref{eq:exS1} implies that
\begin{equation}\label{eq:exS1d} 
\setlength{\arrayrulewidth}{.06em}
\makeatletter
\def\y@vr{\vrule height0.75\y@b@xdim width\y@linethick depth 0.3\y@b@xdim}
\def\y@setdim{%
  \ify@autoscale%
   \ifvoid1\else\typeout{Package youngtab: box1 not free! Expect an
     error!}\fi%
   \setbox1=\hbox{A}\y@b@xdim=1.8\ht1 \setbox1=\hbox{}\box1%
  \else\y@b@xdim=\y@boxdim \advance\y@b@xdim by -2\y@linethick
  \fi}
\makeatother
 e(t_\lambda) 
 \widetilde{g}_1 = 
\sum_{\theta \in I(\lambda)} 
\widetilde{e} \bigl(\,
 \young(\tto\ttt\ttd\ttf\tts\ttn\exu\exv,\tto\ttt\ttd\ttf\tts\ttn\exw\exx)\, \bigr)
\theta \sgn(\theta) + \widetilde{R}.  \end{equation}
Working with $(8,8)$-tableaux with entries 
from the set of formal symbols 
$\{1,2,3,4,5, 6, \onep,\twop,\threep, \linebreak \fourp,\fivep,\sixp, \alpha,
\beta, \gamma, \delta \}$,
the (dual) Garnir relation
\begin{align*}
\widetilde{e} \bigl( \, \tyoung{(123456\alpha\beta,\onep\twop\threep\fourp\fivep\sixp\gamma\delta)}\,
\bigr) 
&= -\widetilde{e} \bigl( \, \tyoung{(123456\alpha\delta,\onep\twop\threep\fourp\fivep\sixp\gamma\beta)}\,
\bigr) 
- \widetilde{e} \bigl( \, \tyoung{(123456\alpha\gamma,\onep\twop\threep\fourp\fivep\sixp\beta\delta)\,
\bigr)}  \\ 
& - \widetilde{e} \bigl( \, \tyoung{(123456\alpha\sixp,\onep\twop\threep\fourp\fivep\beta\gamma\delta)}\,
\bigr) - \cdots - 
\widetilde{e} \bigl( \, \tyoung{(123456\alpha\onep,\beta\twop\threep\fourp\fivep\sixp\gamma\delta)}\,
\bigr)
\end{align*}
holds in $\widetilde{S}^{(8,8)}$.
Applying this relation to $\widetilde{e}(\mathbf{S}_1^\dagger)$, as it appears in~\eqref{eq:exS1d},
and using the definition of $\mathbf{S}_2^\dagger$ and $\mathbf{S}_3^\dagger$ in~\eqref{eq:exts},
 we get
\[ \widetilde{e}(\mathbf{S}_1^\dagger) = -\widetilde{e}(\mathbf{S}_2^\dagger)\theta_2 
- \widetilde{e}(\mathbf{S}_3^\dagger)\theta_3 -
\widetilde{e}(\mathbf{T}_1)- \cdots - \widetilde{e}(\mathbf{T}_6)\]
for some $\theta_2, \theta_3 \in I(\lambda)$ and some $\mathbf{T}_1, \ldots, \mathbf{T}_6\in
\NOmega(\Omega^\lambda)$. Note that 
in the first row of each $\mathbf{T}_j$ there are two tableaux $t$ and $t^\star$
with entries from $\Omega^\lambda$ such that the deindexed tableaux $D(t)$ and $D(t^\star)$ are equal.
For each $j$, let $\psi_j \in I(\lambda)$ be the unique permutation that is the  product of three disjoint
transpositions such that $t\psi_j = t^\star$.  Thus 
$\widetilde{e}(\mathbf{T}_j)(1 + \sgn(\psi_j) \psi_j)  = 0$. Hence 
\[\!\!\sum_{\theta \in I(\lambda)} \!\!\!\widetilde{e}(\mathbf{S}_1^\dagger) \theta \sgn(\theta) = \!
-\sgn(\theta_2) \!\!\sum_{\theta \in I(\lambda)} \!\!\widetilde{e}(\mathbf{S}_2^\dagger) \theta \sgn(\theta) 
-\sgn(\theta_3) \!\!\sum_{\theta \in I(\lambda)} \!\!\widetilde{e}(\mathbf{S}_3^\dagger) \theta \sgn(\theta)  \]
and so $e(t_\lambda) 
\widetilde{g}_1 = -\sgn(\theta_2) e(t_\lambda) 
\widetilde{g}_2 - \sgn(\theta_3) e(t_\lambda)
\widetilde{g}_3$. It follows that the
homomorphisms $\widetilde{g}_1, \widetilde{g}_2, \widetilde{g}_3$
are linearly dependent.

Therefore 
the multiplicity of $S^{(4^4,3^5,2^5,1^7)}$ in
$H_{(2,1)}^{(2^8)}$ is at most $4$. 
Calculation using symmetric functions
in the computer algebra package {\sc Magma}~\cite{Magma}
shows that in fact the multiplicity is exactly $4$.

\end{example}

\section{Lexicographically maximal and minimal constituents}
\label{sec:applications}
In this section we determine
the lexicographically minimal and maximal partitions 
labelling summands of the generalized Foulkes modules $H_\mu^\nu$.
This problem was addressed by Agaoka~\cite{Agaoka}, in the context of plethysms
of symmetric functions, who made sixteen conjectures on the form of such partitions. 
It is remarkable that although tableaux (as opposed to Young diagrams) never formally
appear in \cite{Agaoka}, all of these conjectures are correct. 

The following
description of the lexicographically greatest constituent \cite[Conjecture~1.2]{Agaoka} was first proved by Iijima in \cite[Theorem~4.2]{Iijima}. We give an alternative proof using 
Theorem~\ref{thm:mainMaximal}.

\begin{corollary}\label{cor:lexMax} Let $\mu$  be a partition of $m$ with $k$ parts and let $\nu$ be a
 partition of $n$ with $\ell$ parts.
The lexicographically greatest partition labelling a Specht module occurring 
as a summand of $H_\mu^\nu$ is 
\[
\left(n\mu_1, \ldots, n \mu_{k-1}, n(\mu_k-1)+\nu_1, \nu_2, \ldots, \nu_\ell \right).
\]
Moreover, this Specht module appears with multiplicity~$1$.
\end{corollary}
\begin{proof}
By Theorem~\ref{thm:mainMaximal}  it suffices to find the lexicographically maximal
weight of a conjugate-semistandard  tableau family tuple of shape $\eta^{\nu'}$ where $\eta = \mu'$.
Conjugating $\mu'$-tableaux, such a tuple becomes
$(\mathcal{T}_1, \ldots, \mathcal{T}_{c})$ where $c$ is the first part of $\nu$
and each $\mathcal{T}_i$ is a set
of $\nu_i'$ distinct \emph{semistandard} $\mu$-tableaux.
To maximise the weight in the lexicographic order, we maximise the number of $1$s occurring as entries
in these $\mu$-tableaux, then the number of $2$s, and so on. 
Hence we take $\mathcal{T}_i=\{t_1, t_2, \ldots, t_{\nu_i'}\}$ where $(a,b)t_i= a$ 
if $a < k$ or 
$j < \mu_k$ and $(k,\mu_k)t_i= k+i-1$. 
The weight of this family is 
$(n\mu_1, \ldots, n \mu_{k-1}, n(\mu_k-1)+\nu_1, \nu_2, \ldots, \nu_\ell)$.
Since there is a unique family of this weight, 
Theorem~\ref{thm:toFamily} implies that the multiplicity is~$1$.
\end{proof}

The description of the lexicographically minimal partition is more complicated. 
It is stated in the following two corollaries of Theorem~\ref{thm:main}, which
prove Conjectures~2.1 and~4.1 of \cite{Agaoka}. We need the following definition.

\begin{definition}
We define the \emph{join} of partitions $\lambda$ and $\tilde{\lambda}$, to be the partition 
whose multiset of parts is the union of the multisets of parts of $\lambda$ and $\tilde{\lambda}$. We denote the join of $\lambda$ and $\tilde{\lambda}$ by $\lambda \join \tilde{\lambda}$.
\end{definition}

For example, the join of $(4,2,1,1)$ and $(6,2,2,1)$ is $(6,4,2^3,1^3)$.

\begin{corollary}\label{cor:concatenate} 
The lexicographically least partition labelling a Specht module occurring  as a summand of $H_\mu^\nu$ is
obtained by taking the join of the lexicographically least partitions labelling Specht modules occurring
as summands of $H_\mu^{(\nu_i)}$ if $m$ is odd or of $H_\mu^{(1^{\nu_i})}$ if $m$ is even.
\end{corollary}
\begin{proof}
This is immediate from Theorem~\ref{thm:main} since the lexicographically minimal  type of a 
conjugate-semistandard tableau family tuple occurs when each conjugate-semistandard tableau family within the tuple has lexicographically minimal type.
\end{proof}

For example, the conjugate-semistandard tableau families $\{s_1,s_2,s_3,s_4\}$ and $\{s_1,s_2,s_4,s_5\}$ 
seen in Examples~\ref{ex:ex} and~\ref{ex:excont} have lexicographically minimal type $(3^2,2^2,1^2)$
and so
the lexicographically least partition labelling a Specht module occurring as a summand of $H_{(2,1)}^{(4,1)}$
is $(3^2,2^2,1^2) \join (2,1) = (3^2,2^3,1^3)$. By Example~\ref{ex:excont}, this summand has
multiplicity~$2$.

It remains to describe the lexicographically minimal type of a conjugate-semistandard tableau family of shape $\mu^n$. For this need a final pre-order:
if $s$ and $t$ are conjugate-semistandard $\mu$-tableaux with
multisets of entries $A$, $B \subseteq \N$, then we 
set $s \preceq_\entry t$ if $A \le B$ in the colexicographic order on subsets on $\N$,
defined in \S\ref{subsec:totalOrder}. It is clear that
a conjugate-semistandard tableau family of shape $\mu^n$ is of lexicographically minimal type if 
and only if it is the initial segment of length $n$ under a total order refining $\preceq_\entry$.

\newcommand{\lesseqdot}{\mathrel{\mathord{\le}\!\!\raisebox{0.875pt}{$\cdot$}}}

\begin{example}\label{ex:lex}
Let $\mu = (3,1)$. Recall that $\le$ denotes the total order on 
conjugate-semistandard $\mu$-tableaux
defined in Definition~\ref{defn:colorder}. Let $\lesseqdot$ 
be the total order refining $\preceq_\entry$ under which,
if $s$ and $t$ are conjugate-semistandard 
$\mu$-tableaux with the same multiset of entries, $s \lesseqdot t$
if and only if $s \le t$.
For example, the initial segment of $\lessdot$ of length $10$ is
\[\begin{split}
\ssyoung{(123,1)}\lessdot\ssyoung{(123,2)}{}&{}\lessdot\ssyoung{(123,3)}\lessdot\ssyoung{(124,1)}\lessdot\ssyoung{(124,2)}\\
&\qquad\lessdot \ssyoung{(134,1)}\lessdot \ssyoung{(123,4)}\lessdot
\ssyoung{(124,3)}\lessdot\ssyoung{(134,2)}\lessdot\ssyoung{(234,2)}\, . 
\end{split}
\] 
The tableaux in positions $7$, $8$ and $9$ have equal multisets of entries,
so their order in the initial segment
depends on our choice of $\lesseqdot$ to refine $\preceq_\entry$.
Thus if $n \le 10$ then there is a unique conjugate-semistandard 
tableau family of shape $(3,1)^n$ and lexicographically minimal type, except when $n=7$ or $n=8$,
in which case there are three.
\end{example}

We now give an algorithm that, given $\mu$ a partition of $m$ and a positive integer $n$,
outputs all initial segments of length $n$ of the total orders
refining $\preceq_\entry$. 
The freedom to choose the refinement of $\preceq_\entry$
enters only in Step~F of the algorithm.
The special case $\mu = (m)$ of this algorithm is a well known method for
finding initial segments of the colexicographic order on sets: see for instance \cite[page 25]{Bollobas}.

Let $\mathrm{CS}(\mu, k)$ denote the set of conjugate-semistandard $\mu$-tableaux 
with entries taken from $\{1,2,\ldots, k\}$. We 
write $\mu \to_c\theta$ to mean that $[\theta]$ is obtained from $[\mu]$ by removing $c$ boxes, no two lying in the same row.

\begin{algorithm}\label{alg:lex}Perform Steps $1$ up to $m$, then Step F.
\begin{itemize}
\item[$\bullet$]
\emph{[Step 1]} Choose
$k_1$ maximal such that $|\CS(\mu, k_1)| \le n$. Let 
$\mathcal{T}_{(1)} = \CS(\mu, k_1)$. 

\item[$\bullet$]
\emph{[Step $j$ for $j \in \{2,\ldots, m\}$]}
Let $(k_1, \ldots, k_{j-1}) = (\ell_1^{c_1}, \ldots, \ell_q^{c_q})$
where the $\ell_i$ are distinct. Choose $k_j \in \N_0$ maximal such that
\[ \sum |\CS(\theta, k_j)| \le n - (|\mathcal{T}_{(1)}| + \cdots + |\mathcal{T}_{(j-1)}|) \]
where
the sum is over all sequences of partitions 
$(\theta^{(1)}, \ldots, \theta^{(q)})$ such that
\begin{equation}
\label{eq:removalSequence}
\mu
\to_{c_1} \theta^{(1)} \to_{c_2} 
\ldots \to_{c_{q-1}}\theta^{(q-1)} \to_{c_q} \theta^{(q)} = \theta.
\end{equation}
For each sequence of partitions, take all $\mu$-tableaux $t$ such that
(i) $t$ has entries $\ell_i+1$  in the positions of $\theta^{(i)}/\theta^{(i-1)}$,
for each $i \in \{1,\ldots, q\}$, \emph{and}~(ii)~$t$ has an element of $\CS(\theta,k_j)$ in the positions  of $\theta$.
Let $\mathcal{T}_{(j)}$ be the set of such tableaux. 

\item[$\bullet$]
\emph{[Step F]}
Let $(k_1,\ldots, k_{m}) = (\ell_1^{d_1}, \ldots, \ell_{r}^{d_r})$.
Let $\mathcal{S}$ be the union over all sequences
$\mu \to_{d_1} \theta^{(1)} \to_{d_2} \theta^{(2)} \to_{d_{m}} \theta^{(r-1)} \to_{d_r} 
\theta^{(r)} = \varnothing$
of the set of $\mu$-tableaux constructed as in (i), by putting $\ell_i+1$ in the positions
of $\theta^{(i)} / \theta^{(i-1)}$ for each $i \in \{1,\ldots, r\}$.
The required initial segments are precisely the sets
\[ \mathcal{T}_{(1)} \cup \cdots \cup \mathcal{T}_{(m)} \cup \mathcal{U} \] 
where
$\mathcal{U}$ is any subset of $\mathcal{S}$ of size $n-(|\mathcal{T}_{(1)}| +
\cdots + |\mathcal{T}_{(m)}|)$.
\end{itemize}
\end{algorithm}

Since $\CS(\theta, 0) = \varnothing$ whenever $\theta$ is non-empty, 
there is always a suitable choice of $k_j$ in each Step $j$. 
By the maximality of $k_j$, the sum
$|\CS(\theta,k_j+1)|$ over 
all sequences as in~\eqref{eq:removalSequence} exceeds 
$n - (|\mathcal{T}_{(1)}| + \cdots + |\mathcal{T}_{(j-1)}|)$. The analogous sum
of $|\CS(\phi,k_j+1)|$
over all sequences with a further step $\theta^{(q)} \to_1 \theta^{(q+1)} = \phi$,
if $k_j > k_{j-1}$, or an extra box removed in \hbox{$\theta^{q-1} \to_{c_{q-1} + 1} \theta^{(q)} = \phi$}
if $k_j = k_{j-1}$, is at least as great, since putting $k_j+1$ in the 
boxes removed in the new final step gives all tableaux counted by the original sum
(and possibly some further tableaux that do not have strictly increasing rows).
The sequence $(k_1,k_2,\ldots, k_{m})$ is therefore weakly decreasing
and the construction in (i) and~(ii)  gives
conjugate-semistandard $\mu$-tableaux, with minimum possible maximal entries.
Similarly, the maximality of $k_{m}$ implies that $|\mathcal{S}| \ge n-(|\mathcal{T}_{(1)}|
+ \cdots + |\mathcal{T}_{(m)}|)$.
Thus $\mathcal{S}$ is the set of conjugate-semistandard $\mu$-tableaux having
exactly $d_i$ entries equal to $\ell_i+1$ for each $i \in \{1, \ldots, r\}$.
Therefore Algorithm~\ref{alg:lex} constructs the required initial segments.

The computer software mentioned earlier includes an implementation of the algorithm.

\begin{example}\label{ex:lex2}
Take $\mu = (3,1)$ and $n = 7$. In Step 1, since 
$\bigl|\CS\bigl((3,1),3\bigr)\bigr| = 3 \le 7$, 
while $|\CS\bigl((3,1),4\bigr)\bigr| = 15 > 7$,
we take $k_1 = 3$ and
\[
\mathcal{T}_{(1)}=\left\{\, \young(123,1)\, ,\, \young(123,2)\, ,\,\young(123,3)\, \right\}.
\]
In Step 2, there are two sequences of partitions to consider: $(3,1) \to_1 (3)$
and $(3,1) \to_1 (2,1)$. Since 
$\bigl|\CS\bigl((3), k\bigr)|+ \bigl|\CS\bigl((2,1), k\bigr)\bigr|$ is $2$
when $k=2$ and~$9$ when $k=3$, we take $k_2 = 2$ and
\[ \mathcal{T}_{(2)} = \left\{\, \young(124,1)\, , \, \young(124,2)\, \right\}. \]
In Step 3, there are three sequences of partitions to consider:
$(3,1) \to_1 (3) \to_1 (2)$, $(3,1) \to_1 (2,1) \to_1 (2)$ and $(3,1) \to_1 (2,1) \to_1 (1,1)$.
The corresponding sum is $1$ if $k_3 = 1$ and $5$ if $k_3 = 2$, so we take $k_3 = 1$
and
\[ \mathcal{T}_{(3)} = \left\{\, \young(134,1)\, \right\}. \]
In Step 4 there are again
three sequences of partitions to consider:
$(3,1) \to_1 (3) \to_1 (2) \to_1 (1)$, $(3,1) \to_1 (2,1) \to_1 (2) \to_1 (1)$ 
and $(3,1) \to_1 (2,1) \to_1 (1,1) \to_1 (1)$.
The corresponding sum is $0$ if $k_4 = 0$ and $3$ if $k_4=1$, so we take $k_4 = 0$ and $\mathcal{T}_{(4)}=
\varnothing$.
In the final step, Step F, we have
\[ \mathcal{S} = \left\{\, \young(123,4)\, , \, \young(124,3)\, , \, \young(134,2)\, \right\}. \]
(Note these are precisely the tableaux obtained by instead taking $k_4 = 1$ in Step 4.)
As expected from Example~\ref{ex:lex}, any tableau in $\mathcal{S}$ may be chosen to complete
an initial segment of length $7$ in a total order refining $\preceq_\entry$; the output
of the algorithm is the three tableaux families $\mathcal{T}_{(1)} \cup
\mathcal{T}_{(2)} \cup \mathcal{T}_{(3)} \cup \mathcal{T}_{(4)} \cup \{t \}$,
where $t \in \mathcal{S}$. 
\end{example}

The following corollary is Conjecture~4.2 in \cite{Agaoka}. It has a
constructive proof using Algorithm~\ref{alg:lex}.

\begin{corollary}\label{cor:lexMin} 
With the notation as above, for each $j \in \{1,\dots, m\}$, 
let $a_j = n - (|\mathcal{T}_{(1)}| + \cdots + |\mathcal{T}_{(j)}|)$ 
and let $b_j = |\mathcal{T}_{(j)}| (m+1-j)/k_j$.
The lexicographically least partition labelling a Specht module 
occurring as a summand of $H_\mu^{(n)}$ 
if~$m$ is odd or of $H_\mu^{(1^{n})}$ if $m$ is even is
\[
\bigl( (k_1+1)^{a_1}, k_1^{b_1-a_1}, (k_2+1)^{a_2}, k_2^{b_2-a_2},\ldots, 
(k_{m}+1)^{a_{m}}, k_{m}^{b_{m}-a_{m}} \bigr)
\]
where it may be necessary to reorder and regroup the parts to form a partition.
\end{corollary}

\begin{proof}
Algorithm~\ref{alg:lex} 
constructs
a conjugate-semistandard tableau family tuple of shape $\mu^n$ and lexicographically least type. 
It remains to show that its type is as claimed. 
For $i\in\{1,2,\ldots,k\}$ and $\eta$ any partition,
the total number of occurrences of $i$ as an entry in a member 
of $\mathrm{CS}(\eta, k)$ is $|\mathrm{CS}(\eta, k)||\eta|/k$.
(This is essentially the statement that the Schur function $s_\eta$ is symmetric.)
Hence  the type of the family $\mathcal{T}_{(1)}$ is $k_1^{b_1}$.
Now consider the entry $k_1+1$. It appears in a removable box of each 
tableau in
$\mathcal{T}_{(2)} \cup \ldots \cup \mathcal{T}_{(m)} \cup \mathcal{U}$.
(If $k_1 = k_2$ there may be other entries in these tableaux equal to $k_1+1$;
these will be counted shortly.) Since there are $a_1 = n - |\mathcal{T}_{(1)}|$
such tableaux, the contribution to the type of the initial segment from
entries considered so far is $\bigl( (k_1+1)^{a_1}, k_1^{b_1-a_1} \bigr)$.
This gives the first two terms.

Let $j \in \{2,\ldots, m\}$.
By the construction in (i) in Step $j$, the number of occurrences of any  
$i\in\{1,2,\ldots,k_j\}$ in the tableaux in $\mathcal{T}_{(j)}$ is 
\[
\sum  \frac{|\theta| |\mathrm{CS}(\theta, k_j)|}{k_j}= \frac{|\mathcal{T}_{(j)}|(m+1-j)}{k_j}=b_j,
\]
where the sum is over all sequences~\eqref{eq:removalSequence}.
Similarly to the case $j=1$, we count one appearance of $k_j+1$ in 
each of the $a_j$ tableaux
in $\mathcal{T}_{(j+1)} \cup \ldots \cup \mathcal{T}_{(m)} \cup \mathcal{U}$.
The contribution to the type of the initial segment is therefore 
$\bigl( (k_j+1)^{a_j}, k_j^{b_j-a_j} \bigr)$.

All entries of the tableaux in $\mathcal{U}$ are now accounted for, so the
type of the initial segment is as claimed.
\end{proof}

By Theorem~\ref{thm:toFamily}, if
there is a unique conjugate-semistandard tableau family of the lexicographically minimal type
then the multiplicity of the lexicographic minimal Specht module is~$1$. Conversely, if
there are two or more such families then we have a choice of which tableaux to choose from
the final set $\mathcal{S}$ in Algorithm~\ref{alg:lex}: the unique greatest element $u$
under $\preceq_\col$ is obtained by putting the entries $\ell_1+1$ as far to the right as possible,
then doing the same with the entries $\ell_2+1$, and so on. Thus $t \prec_\col u$ for all
$t \in \mathcal{U}$ with $t\not= u$, and taking subsets $\mathcal{U}$ and $\mathcal{U}'$ of $\mathcal{S}$
of the appropriate size such that $u \in \mathcal{U}$ and $u\not\in \mathcal{U}'$ 
gives families satisfying the hypotheses for Proposition~\ref{prop:linindep}, 
and so the multiplicity is at least~2.

In general the multiplicity of the lexicographically minimal Specht module may be arbitrarily large.

\begin{corollary}\label{cor:minMult}
Let $m \in \N$. Let $n \in \N$ and let $\nu = (1^n)$ if $m$ is even
and $(n)$ is $m$ is odd.
The multiplicity of the lexicographically minimal constituent
of $H_{(m-1,1)}^{\nu}$ is equal to $\binom{m-1}{e}$ for some $e \in \{1,\ldots, m-1\}$.
Moreover all these values are attained for some $n \in \N$. 
\end{corollary}

\begin{proof}
If the sequence $(k_1,\ldots,k_m)$ created by Algorithm~\ref{alg:lex} has a
repeated element, say $k$, then since there is a unique conjugate-semistandard $(m-1,1)$-tableau
with two occurrences of $k+1$, the final set $\mathcal{S}$ is a singleton. 
Otherwise $\mathcal{S}$ consists of all $m-1$ standard Young tableaux with content 
$\{k_1+1,\ldots,k_m+1\}$. In this case there are $\binom{m-1}{e}$ conjugate-semistandard
tableau families of shape $(m-1,1)^n$, where $e = n - (|\mathcal{T}_{(1)}| + \cdots
+ |\mathcal{T}_{(m)}|)$. 

Conversely, let $\lessdot$ be the total order refining $\preceq_\entry$ defined 
in the same way as Example~\ref{ex:lex}. Let $h_1 > h_2 > \ldots > h_m$, and 
for $h \in \{h_1,\ldots, h_{m-1}\}$, let
$t(h)$ be the unique conjugate-semistandard
$(m-1,1)$-tableau with content $\{h_1, h_2, \ldots, h_m \}$ having
$h$ in position $(2,1)$.
There is an initial segment of $\lessdot$ 
ending
\[ t(h_1) \lessdot t(h_2) \lessdot \ldots \lessdot t(h_{m-1}). \]
 (When this initial segment is
constructed using Algorithm~\ref{alg:lex}, $k_i + 1 = h_i$ for each $i \in \{1,\ldots, m\}$.)
Therefore, given any $e \in \{1,\ldots, m-1\}$
there exists $n$ such that there are exactly $\binom{m-1}{e}$ conjugate-semistandard
tableau families of shape $(m-1,1)^n$ and lexicographically minimal type, corresponding to the $\binom{m-1}{e}$ choices in Step F.
Suppose that $\mathcal{T}^{(1)}, \ldots, \mathcal{T}^{(d)}$ 
are these families and let $\lambda$ be their common type. 
Since the columnar tableaux $||t(h_1)||$, $||t(h_2)||$, \ldots, $||t(h_{m-1})||$ are totally ordered
under $\preceq_\col$ with $||t(h_1)|| \preceq_\col ||t(h_2)|| \preceq_\col \ldots 
\preceq_\col ||t(h_{m-1})||$,
we see that the hypotheses for Proposition~\ref{prop:linindep}
are satisfied. We therefore have $\binom{m-1}{e}$ linearly independent homomorphisms
from $S^\lambda$ to $H_{(m-1,1)}^\nu$.
By Theorem~\ref{thm:toFamily}, the multiplicity is exactly~$\binom{m-1}{e}$.
\end{proof}

We remark that the least $n$ such that the 
multiplicity in Corollary~\ref{cor:minMult} attains its maximum value of 
$\binom{m-1}{\lfloor(m-1)/2\rfloor}$ is $m^2/2 -1$ if $m$ is even, and $(m^2-1)/2$ if $m$ is odd.

For the uniqueness result mentioned at the end of
the introduction we need the following proposition.
Recall that a partition of the form $(a^b)$ for $a$, $b \in \N$ is said to be
\emph{rectangular}.

\begin{proposition}\label{prop:unique}
There is a unique conjugate-semistandard tableau family of shape $\mu^n$ and maximal weight
if and only if $m=1$, or $n=1$, or $\mu$ is rectangular and $n = 2$.
\end{proposition}

\begin{proof}
Let $t$ be the $\mu$-tableau such that $(a,b)t = b$ for all $(a,b) \in [\mu]$. 
Let~$e$ be least such that row $e$ of $[\mu]$ 
has a removable box. The box $(e,\mu_e)$ of $t$ contains $\mu_e$; this is the greatest number
appearing in $t$. 
For $c \in \N$, let $t^{+c}$ be the conjugate-semistandard tableau obtained from $t$ by replacing this entry
with $\mu_e + c$. Let $\eta = \mu'$. Note that $\eta$ has $\mu_e$ parts, with $\eta_{\mu_e} = e$.
The weight of the conjugate-semistandard tableau family  $\{t, t^{+1}, \ldots, t^{+(n-1)} \}$
is the partition $\lambda_{\max}$ 
such that $[\lambda_{\max}]$ is obtained from $[n \eta]$ 
by removing $n-1$ boxes from row $\mu_e$ (these correspond to the entries equal to $\mu_e$
changed in the definition of $t^{+c}$) and inserting one box
in each of the previously empty rows $\mu_e+1, \ldots, \mu_e+n-1$.
Since the first $\mu_e-1$ rows of $[n\eta]$ are unaffected, it is clear
that $\{t, t^{+1}, \ldots, t^{+(n-1)} \}$ is the conjugate-semistandard tableau family
of shape $\mu^n$ and lexicographically
maximal weight. (This is a special case of Corollary~\ref{cor:lexMax} and \cite[Theorem~4.2]{Iijima}.)

Assume that $m \not= 1$.
If $n \ge 3$ or $\mu$ is not a rectangle, then the conjugate-semistandard
tableau family of lexicographically minimal type has no entry equal to $\mu_e+n-1$.
Its weight has fewer parts than $\lambda_{\max}$, and so 
is incomparable with $\lambda_{\max}$ in the dominance order.
\end{proof}

We give a brief example. The conjugate-semistandard tableau family
of shape $(3,3)^3$ and lexicographically maximal weight is
\[ \Bigl\{ \, \young(123,123)\, , \quad \young(123,124)\, , \quad \young(123,125)\, \Bigr\}. \]
As seen in the proof of Proposition~\ref{prop:unique}, 
its weight is $(6,6,4,1,1)$, obtained from $[3\mu'] = [(6,6,6)]$ by removing
two boxes from the lowest row, and inserting new boxes in rows $4$ and $5$.
The conjugate-semistandard tableau family of lexicographically minimal type is
\[ \Bigl\{ \, \young(123,123)\, , \quad \young(123,124)\, , \quad \young(123,134)\, \Bigr\}; \]
it has the incomparable weight $(6,5,5,2)$.

Recall that $\kappa = \nu'$ if $m$ is even and $\kappa = \nu$ if $m$ is odd.

\begin{corollary}\label{cor:unique} Let $m \ge 2$.
\begin{thmlist}
\item There is a unique partition $\lambda$, maximal in the dominance order on partitions,
such that~$S^\lambda$ is a summand of $H_\mu^\nu$ if and only if either $\nu = (n)$ or
$\mu$ is rectangular and $\nu$ has exactly two parts.

\item There is a unique partition $\lambda$, minimal in the dominance order on partitions,
such that~$S^\lambda$ is a summand of $H_\mu^\nu$ if and only if either $\kappa = (1^n)$
or $\mu$ is rectangular and $\kappa = (2^c,1^d)$ for some $c,d \in \N$.
\end{thmlist}
\end{corollary}

\begin{proof}
This follows immediately from  Proposition~\ref{prop:unique} using 
Theorem~\ref{thm:mainMaximal} for~(i) 
and Theorem~\ref{thm:main} for (ii).
\end{proof}

When $\mu = (a^b)$, the unique closed conjugate-semistandard tableau family of shape $\mu^2$ has
type $(a+1,a^{2b-2},a-1)$. Using this it is routine to determine the partitions
$\lambda$ in Corollary~\ref{cor:unique} explicitly. For example, if $m$ is even
then the unique minimal partition $\lambda$ such
that $S^\lambda$ is a summand of $H_{(a^b)}^{(n-c,c)}$ is 
$\bigl( (a+1)^{c},a^{(2b-2)c+b(n-2c)},(a-1)^c \bigr)$.

We end with an observation on invariants of special linear groups.

\begin{corollary}
Let $k \in \N$. Suppose that $m$ is odd 
and $n = |\CS(\mu,k)|$. 
Let $d = mn/k$ and let $E$ be a $b$-dimensional complex vector space. There is a unique polynomial
invariant of degree $n$ for the action of $\SL(E)$ on $\nabla^{\mu}(E)$.
\end{corollary}

\begin{proof}
As seen in the proof of Corollary~\ref{cor:lexMin},
the conjugate-semistandard tableau family of shape $\mu^n$
consisting of all tableaux with greatest
entry at most $k$ has lexicographically minimal type $(k^d)$. 
(This proves \cite[Conjecture~2.2]{Agaoka}.) 
By the remark following Corollary~\ref{cor:lexMin}, the multiplicity of $S^{(k^d)}$ in
$H_\mu^{(n)}$ is $1$. Hence $\nabla^{(k^d)}(E)$ has multiplicity $1$ in $\Sym^n \nabla^\mu (E)$.
Since 
\[ \nabla^{(k^d)}(E) \cong \bigl(\bigwedge^d E\bigr)^{\otimes k} \cong \det{}^{\!k},\] 
the result follows.
\end{proof}

\section*{Acknowledgements}

The authors thank two anonymous referees for their careful reading of an earlier version
of this paper.

\def\cprime{$'$} \def\Dbar{\leavevmode\lower.6ex\hbox to 0pt{\hskip-.23ex
  \accent"16\hss}D} \def\cprime{$'$}
\providecommand{\bysame}{\leavevmode\hbox to3em{\hrulefill}\thinspace}
\providecommand{\MR}{\relax\ifhmode\unskip\space\fi MR }
\renewcommand{\MR}[1]{\relax} 
\providecommand{\MRhref}[2]{%
  \href{http://www.ams.org/mathscinet-getitem?mr=#1}{#2}
}
\providecommand{\href}[2]{#2}
\renewcommand{\bibname}{\relax}

\end{document}